\documentclass[onecolumn,final,a4paper,10pt,sort&compress]{elsarticle}

\makeatletter
\def\ps@pprintTitle{%
 \let\@oddhead\@empty
 \let\@evenhead\@empty
 \def\@oddfoot{}%
 \let\@evenfoot\@oddfoot}
\makeatother

\usepackage{type1cm}          

\usepackage{makeidx}          
\usepackage{graphicx}         

\usepackage{multirow}         
\usepackage{multicol}         
\usepackage[bottom]{footmisc} 

\usepackage{newtxtext}        %
\usepackage{newtxmath,bm,bbm} 
\usepackage{setspace}
\usepackage{placeins,subfigure,overpic,rotating,subfigure}
\usepackage{braket}
\usepackage{chngpage}
\usepackage{cases}
\usepackage{tabularx}
\usepackage{multirow}

\usepackage{lettrine}
\usepackage{booktabs}

\usepackage{setspace}
\usepackage{comment}

\newcommand{\Eh} {\mathcal{E}_h}

\newcommand{\vrtx}{\mathsf{v}}

\newcommand{\fh}{f_h}

\usepackage{color}
\definecolor{MyGray}{rgb}{0.61,0.61,0.61}
\definecolor{MyDarkGreen}{rgb}{0,0.45,0}

\newcommand{\RED}[1]{{\color{red}#1}}

\def\trait #1 #2 #3 {\vrule width #1pt height #2pt depth #3pt}
\def\fin{\hfill
        \trait .3 5 0
        \trait 5 .3 0
        \kern-5pt
        \trait 5 5 -4.7
        \trait 0.3 5 0
\medskip}
\newcommand{\ENDPROOF}{\fin}

\newtheorem{theorem}{Theorem}[section]
\newtheorem{lemma}[theorem]{Lemma}

\newtheorem{example}[theorem]{Example}
\newtheorem{remark}[theorem]{Remark}
\newenvironment{proof}{\textit{Proof.}}{\fin}

\def\P{{\mathbb P}}

\newcommand{\Vhrs}[1]{{(V^{\ptwo,\pone}_{h,#1})}^*}

\newcommand{\Vhrp} [1]{V^{\ptwo,\pone}_{h,#1}}
\newcommand{\VhPrp}[1]{V^{\ptwo,\pone}_{h,#1}(\P)}
\newcommand{\VhtPrp}[1]{\widetilde{V}^{\ptwo,\pone}_{h,#1}(\P)}

\newcommand{\Whrp} [1]{W^{\ptwo,\pone}_{h,#1}}
\newcommand{\WhPrp}[1]{W^{\ptwo,\pone}_{h,#1}(\P)}

\newcommand{\INTP}{\footnotesize{I}}

\newcommand{\REAL}{\mathbbm{R}}

\newcommand{\CARD}[1]{\#{\big(#1\big)}}
\newcommand{\ASSUM}[1]{\textbf{(#1)}}
\newcommand{\TERM} [1]{\mbox{$\mathbf{(#1)}$}}
\newcommand{\tTERM}[1]{\begin{tiny}\mbox{$\mathbf{(#1)}$}\end{tiny}}
\newcommand{\TABROW}{}

\newcommand{\EOD}{\end{document}}

\newcommand{\FNOTE}[1]{\footnote{\textbf{#1}}}

\newcommand{\tell}{\widetilde{\ell}}

\newcommand{\nv}{\mathbf{n}}

\newcommand{\tv}{\mathbf{t}}

\newcommand{\xv}{\mathbf{x}}
\newcommand{\yv}{\mathbf{y}}

\newcommand{\as}{a}

\newcommand{\cs}{c}

\newcommand{\fs}{f}

\newcommand{\hs}{h}

\newcommand{\js}{j}
\newcommand{\ks}{k}

\newcommand{\ms}{m}

\newcommand{\ps}{p}
\newcommand{\qs}{q}
\newcommand{\rs}{r}
\renewcommand{\ss}{s}
\newcommand{\us}{u}
\newcommand{\vs}{v}
\newcommand{\ws}{w}
\newcommand{\xs}{x}
\newcommand{\ys}{y}

\newcommand{\Cs}{C}
\newcommand{\Ds}{D}

\newcommand{\Ms}{M}

\newcommand{\Vs}{V}

\newcommand{\Vsp}{V^{\prime}}

\newcommand{\calD}{\mathcal{D}}
\newcommand{\calE}{\mathcal{E}}

\newcommand{\calP}{\mathcal{P}}

\newcommand{\calV}{\mathcal{V}}

\newcommand{\LTWO}  {L^2}

\newcommand{\LS}[1] {L^{#1}}
\newcommand{\HS}[1] {H^{#1}}

\newcommand{\CS}[1] {C^{#1}}

\newcommand{\PS}[1] {\mathbbm{P}_{#1}}

\renewcommand{\P} {\textsf{P}}            
\newcommand  {\E} {e}
\newcommand  {\V} {\mathsf{v}}

\newcommand{\hh}{h}
\newcommand{\Th}{\Omega_{\hh}}

\newcommand{\xvP}{\xv_{\P}}        
\newcommand{\xvE}{\xv_{\E}}        
\newcommand{\xvV}{\xv_{\V}}        

\newcommand{\hP}{\hh_{\P}}

\newcommand{\hE}{\hh_{\E}}
\newcommand{\hV}{\hh_{\V}}

\newcommand{\mP}{\ABS{\P}}

\newcommand{\Pset}{\mathcal{P}}    
\newcommand{\Eset}{\mathcal{E}}    
\newcommand{\Vset}{\mathcal{V}}    

\newcommand{\NMB}{N}
\newcommand{\NP}{\NMB^{\Pset}}   
\newcommand{\NE}{\NMB^{\Eset}}   
\newcommand{\NV}{\NMB^{\Vset}}   

\newcommand{\NPV}{\NMB^{\Vset}_{\P}}      
\newcommand{\NPE}{\NMB^{\Eset}_{\P}}      

\newcommand{\dS}{\,ds}
\newcommand{\dx}{\,d\xv}

\newcommand{\tsx}{t_x}
\newcommand{\tsy}{t_y}

\newcommand{\nsx}{n_x}
\newcommand{\nsy}{n_y}

\newcommand{\fsh}{\fs_{\hh}}

\newcommand{\ush}{\us_{\hh}}

\newcommand{\usI}{\us_{\INTP}}

\newcommand{\vsh}{\vs_{\hh}}

\newcommand{\wsh}{\ws_{\hh}}

\newcommand{\asP}{\as_{\pone}^{\P}}

\newcommand{\ash}{\as_{\hh}}

\newcommand{\ashP}{\as_{\hh}^{\P}}

\newcommand{\cshat}{\widetilde{\cs}}
\newcommand{\muh}  {\widetilde{\mu}}

\newcommand{\BIL} [2]{\left<#1,#2\right>}
\newcommand{\bil} [2]{\big<#1,#2\big>}

\newcommand{\SP} {S^{\P}}

\newcommand{\nlen}{\hspace{-0.2mm}}

\newcommand{\snorm}  [2]{|#1|_{#2}}

\newcommand{\norm}   [2]{|\nlen|#1|\nlen|_{#2}}

\newcommand{\ABS}    [1]{\left|#1\right|}

\newcommand{\PizP}[1]{\Pi^{0,\P}_{#1}}

\newcommand{\PiPr}[1]{\Pi^{\pone,\P}_{#1}} 
\newcommand{\Pizr}[1]{\Pi^{0,\P}_{#1}}  

\newcommand{\restrict}[2]{{#1}_{|{#2}}}

\newcommand{\pone}{\ps_1}
\newcommand{\ptwo}{\ps_2}

\usepackage[a4paper]{geometry}

\begin{document}

\begin{frontmatter}


  \title{On arbitrarily regular conforming virtual element methods for
    elliptic partial differential equations}
  
  
  \author[MOXA] {P.~F.~Antonietti}
  \author[T5]   {G. Manzini}
  \author[UNIMI]{S. Scacchi}
  \author[MOXA] {and M. Verani}


  \address[MOXA]{
    MOX, Dipartimento di Matematica,
    Politecnico di Milano, Italy;
    \emph{e-mail:\{paola.antonietti, marco.verani\}@polimi.it}
  }
  \address[T5]{
    Group T-5,
    Theoretical Division,%
    Los Alamos National Laboratory,
    Los Alamos, NM,
    USA;
    \emph{e-mail: gmanzini@lanl.gov}
  }
  \address[UNIMI]{
    Dipartimento di Matematica,
    Universit\`a degli Studi di Milano,
    Italy,
    \emph{e-mail: simone.scacchi@unimi.it}
  }


  \renewcommand{\RED}[1]{#1}
  \renewcommand{\FNOTE}[1]{} 

  \begin{abstract}
    The Virtual Element Method (VEM) is a very effective framework to
    design numerical approximations with high global regularity to the
    solutions of elliptic partial differential equations.
    In this paper, we review the construction of such approximations
    for an elliptic problem of order $\pone$ using conforming, finite
    dimensional subspaces of $\HS{\ptwo}(\Omega)$, where $\pone$ and
    $\ptwo$ are two integer numbers such that $\ptwo\geq\pone\geq1$
    and $\Omega\in\REAL^2$ is the computational domain.
    An abstract convergence result is presented in a suitably defined
    energy norm.
    The space formulation and major aspects such as the choice and
    unisolvence of the degrees of freedom are discussed, also
    providing specific examples corresponding to various practical
    cases of high global regularity.
    Finally, the construction of the ``enhanced'' formulation of the
    virtual element spaces is also discussed in details with a proof
    that the dimension of the ``regular'' and ``enhanced'' spaces is
    the same and that the virtual element functions in both spaces can
    be described by the same choice of the degrees of freedom.
  \end{abstract}


  
\end{frontmatter}
  

\raggedbottom
\setcounter{secnumdepth}{4}
\setcounter{tocdepth}{4}


\allowdisplaybreaks

\renewcommand{\RED}[1]{#1}
\renewcommand{\FNOTE}[1]{}

\section{Introduction}
\label{sec1:intro}

The conforming finite element method is based on the construction of a
finite dimensional approximation spaces that are typically only
$\CS{0}$-continuous~\cite{Ciarlet:2002} on the meshes covering the
computational domain.
The construction of sych approximation spaces with higher regularity
is normally deemed a difficult task because it requires a set of basis
functions with such global regularity.
Examples in this direction can be found all along the history of
finite elements from the oldest works in the sixties of the last
century,
e.g.,~\cite{Argyris-Fried-Scharpf:1968,Bell:1969,Clough-Tocher:1965}
to to the most recent attempts
in~\cite{Zhang:2009,Zhang:2016,Zhang-Hu:2015,Wu-Lin-Hu:2021}.
Despite its intrinsic difficulty, designing approximations with global
$\CS{1}$- or higher regularity is still a major research topic.
Such kind of approximations have indeed a natural application in the
numerical treatment of problems involving high-order differential
operators.

The Virtual Element Method
(VEM)~\cite{BeiraodaVeiga-Brezzi-Cangiani-Manzini-Marini-Russo:2013}
does not require the explicit knowledge of the basis functions
spanning the approximation spaces in its formulation and
implementation.
The crucial idea behind the VEM is that the elemental approximation
spaces, which are globally ``glued'' in a highly regular conforming
way, are defined elementwise as the solutions of a partial
differential equation.
The functions that belong to such approximation spaces are dubbed as
``virtual'' as they are never really computed, with the noteworthy
exception of a subspace of polynomials that are indeed used in the
formulation of the method.
The virtual element functions are uniquely characterized by a set of
values, the \emph{degrees of freedom}, that are actually solved for in
the method.

The virtual element 'paradigm'' thus provides a major breakthrough in
obtaining highly regular Galerkin methods as it allows the
construction of numerical approximation of any order of accuracy and
global conforming regularity that work on unstructured two-dimensional
and three-dimensional meshes with very general polytopal elements.

Roughly speaking, the VEM is a Galerkin-type projection method that
generalize the finite element method, which was originally designed
for simplicial and quadrilateral/hexahedral meshes, to polytopal
meshes.
Other important families of methods that are suited to polytopal
meshes are
the polygonal/polyhedral finite element method
\cite{Sukumar-Tabarraei:2004}
the mimetic finite difference
method\cite{BeiraodaVeiga-Lipnikov-Manzini:2014}
the discontinuous Galerkin method on polygonal/polyhedral
grids\cite{Antonietti-Giani-Houston:2013,Cangiani-Dong-Georgoulis-Houston:2017};
the hybrid discontinuous Galerkin
method\cite{Cockburn-Dong-Guzman:2008}; and the
hybrid high--order method~\cite{DiPietro-Droniou:2020}.

The conforming VEM was first developed for second-order elliptic
problems in primal
formulation~\cite{BeiraodaVeiga-Brezzi-Cangiani-Manzini-Marini-Russo:2013,BeiraodaVeiga-Brezzi-Marini-Russo:2016b},
and then in mixed
formulation~\cite{BeiraodaVeiga-Brezzi-Marini-Russo:2016c,Brezzi-Falk-Marini:2014}
and nonconforming
formulation~\cite{AyusodeDios-Lipnikov-Manzini:2016}.
Despite its relative youthness (the first paper was published in
2013), the VEM has been very successful in a wide range of scientific
and engineering applications.
A non-exhaustive list of applications includes, for example, the works of References
\cite{Benedetto-Berrone-Borio:2016,Perugia-Pietra-Russo:2016,Mascotto-Perugia-Pichler:2019a,Antonietti-Bruggi-Scacchi-Verani:2017,Chi-Pereira-Menezes-Paulino:2020,Wriggers-Rust-Reddy:2016,Aldakheel-Hudobivnik-Hussein-Wriggers:2018,Certik-Gardini-Manzini-Vacca:2018:ApplMath:journal,Antonietti-Bertoluzza-Prada-Verani:2020,Benvenuti-Chiozzi-Manzini-Sukumar:2019,ArtiolideMirandaLovadinaPatruno:2017,Artioli-deMiranda-Lovadina-Patruno:2018,Park-Chi-Paulino:2019,Park-Chi-Paulino:2020,Antonietti-Manzini-Mazzieri-Mourad-Verani:2021}.
Virtual element spaces forming de Rham complexes for the Stokes,
Navier-Stokes and Maxwell equations were proposed
in~\cite{BeiraodaVeiga-Lovadina-Vacca:2016,%
  BeiraodaVeiga-Lovadina-Vacca:2018,%
  BeiraodaVeiga-Brezzi-Marini-Russo:2016a}.
A VEM for Helmholtz problems based on non-conforming approximation
spaces of Trefftz functions, i.e., functions that belong to the kernel
of the Helmholtz operator, is
found~\cite{Mascotto-Perugia-Pichler:2019}

The first works using a $\CS{1}$-regular conforming VEM addressed the
classical plate bending
problems~\cite{Brezzi-Marini:2013,Chinosi-Marini:2016}, second-order
elliptic
problems~\cite{BeiraodaVeiga-Manzini:2014,BeiraodaVeiga-Manzini:2015},
and the nonlinear Cahn-Hilliard
equation~\cite{Antonietti-BeiraodaVeiga-Scacchi-Verani:2016}.
More recently, highly regular virtual element spaces were considered
for the von {K}\'arm\'an equation modelling the deformation of very
thin plates~\cite{Lovadina-Mora-Velasquez:2019}, geostrophic
equations~\cite{Mora-Silgado:2021} and fourth-order subdiffusion
equations~\cite{subdiffusion:2021}, two-dimensional plate vibration
problem of Kirchhoff plates~\cite{Mora-Rivera-Velasquez:2018}, the
transmission eigenvalue problems~\cite{Mora-Velasquez:2018} the
fourth-order plate buckling eigenvalue
problem~\cite{Mora-Velasquez:2020}.
In~\cite{Antonietti-Manzini-Verani:2019}, we proposed the
highly-regular conforming VEM for the two-dimensional polyharmonic
problem $(-\Delta)^{\pone}\us=\fs$, $\pone\geq 1$.
The VEM is based on an approximation space that locally contains
polynomials of degree $\rs\geq2\pone-1$ and has a global $\HS{\pone}$
regularity.
In~\cite{Antonietti-Manzini-Scacchi-Verani:2021}, we extended this
formulation to a virtual element space that can have arbitrary
regularity $\ptwo\geq\pone\geq1$ and contains polynomials of degree
$\rs\geq\ptwo$.
This VEM is a generalization of the VEMs for second- and fourth-order
problems since the approximation space for $\ptwo=\pone=1$ coincides
with the conforming virtual element spaces for the Poisson equation of
Reference~\cite{BeiraodaVeiga-Brezzi-Cangiani-Manzini-Marini-Russo:2013}
and the approximation space for $\ptwo=\pone=2$ coincides with the
conforming virtual element spaces for the and the biharmonic equation
of Reference~\cite{Brezzi-Marini:2013}.
VEMs for three-dimensional problems are also available for the
fourth-order linear elliptic
equation~\cite{BeiraodaVeiga-Dassi-Russo:2020} (see
also~\cite{Brenner-Sung:2019}), and highly-regular conforming VEM in
any dimension has been proposed in~\cite{Huang:2021}.

\medskip
In this paper, we review the detailed construction of the virtual
element spaces with arbitrary order of accuracy and regularity for the
numerical approximation of two-dimensional problems involving the
polyharmonic operator of degree $\pone$.
Such a construction follows the standard guidelines of the VEM, which
we briefly summarize here.
As the VEM is a conforming Galerkin variational method, its
formulation requires the definition of a suitable finite dimensional
approximation space, which is obtained by combining in a conforming
way local (elemental) finite dimensioal spaces.
The local virtual element spaces are defined in every mesh element by
all the solutions of a specific polyharmonic problem of degree
$\ptwo\geq\pone$.
The loading terms of the partial differential equations defining the
elememtal virtual element spaces can be all the polynomials of degree
(up to) $\rs-2\pone$, where the integer number $\rs\geq\ptwo$ is the
order of the virtual element space.
The traces of the virtual element functions and all its normal
derivatives of order $\js$ from one to $\ptwo-1$ on the elemental
edges are univariate polynomials of degree at least $\rs-\js$ (in some
cases the polynomial degree can be a little higher than $\rs-\js$).
From the definition, it also follows the fundamental property that the
polynomials of degree up to $\rs$ inside all elements are a linear
subspace of the virtual element space of degree $\rs$.
Then, the elemental spaces are ``glued'' together to form a global
space with $H^{\ptwo}$-regularity.
A very careful choice of the degrees of freedom, which as usual are
nodal values associated with the mesh vertices or polynomial moments
associated with edges and elements, makes the elliptic projection onto
the polynomials of degree $\rs$ computable.
An $\LS{2}$-orthogonal projection onto the polynomials of degree
$\rs-\pone$ in every mesh element is also computable in the
``modified'' (or ``enhanced'') formulation of the virtual element
method.
In this work, we also present a detailed discussion of the enhanced
formulation and its major properties.
The enhanced formulation is obtained by extending the similar
contruction for the Poisson equation presenting in the pioneering
paper~\cite{Ahmad-Alsaedi-Brezzi-Marini-Russo:2013} to our case.
These polynomial projection operators are finally used to construct
the discrete approximation of the bilinear form and the right-hand
side that are used in the virtual element approximation.
An abstract convergence result holds, that can be proved by assuming
only a few fundamental properties of the virtual element formulation.

\medskip
The remaining part of the manuscript is organized as follows.
In Section~\ref{sec2:continuous_pbl} we introduce the continuous
polyharmonic problem and its weak formulation.
In Section~\ref{sec3:discrete_pbl} we introduce the virtual element
discretization and recall the main abstract convergence result.
In Section~\ref{sec4:VEM} we present the formulation of the conforming
virtual element approximation with higher-order regularity. 
Finally, in Section~\ref{sec6:conclusions} we draw our conclusions.
\section{The continuous problem}
\label{sec2:continuous_pbl}

Let $\Omega\subset\REAL^2$ be an open, bounded, convex domain with
polygonal boundary $\Gamma$.
For any integer $\pone\geq1$, we consider the polyharmonic problem
\begin{subequations}
  \label{eq:poly:strong}
  \begin{align}
    (-\Delta)^{\pone}\us &= \fs\phantom{0}\qquad\text{in~}\Omega,\label{eq:poly:pblm:1}\\
    \partial^j_n\us   & = 0\phantom{\fs}\qquad\text{for~}j=0,\ldots,\pone-1\text{~on~}\Gamma,\label{eq:poly:pblm:2}
  \end{align}
\end{subequations}
where $\partial_n\us=\nv\cdot\nabla\us$ is the normal derivative of
$\us$ and $\partial^j_n\us$ is the normal derivative applied $j$ times
to $\us$ with the useful convention that $\partial^0_n\us=\us$ for
$j=0$.
Let
\begin{align*}
  \Vs:=\HS{\pone}_{0}(\Omega) =
  \big\{\vs\in\HS{\pone}(\Omega):\partial^j_n\vs=0\text{~on~}\Gamma,\,j=0,\ldots,\pone-1\big\}.
\end{align*}
Denoting the duality pairing between $\Vs$ and its dual $\Vsp$ by
$\BIL{\cdot}{\cdot}$, the variational formulation of the
polyharmonic problem~\eqref{eq:poly:strong} reads as
\begin{align}
  \mbox{\emph{Find $\us\in\Vs$ such that:}}\quad
  \as_{\pone}(\us,\vs) = \BIL{\fs}{\vs} \quad\forall\vs\in\Vs,
  \label{eq:poly:varform}
\end{align}
where, for any nonnegative integer $\ell$, the bilinear form
$\as_{\pone}(\cdot,\cdot):\Vs\times\Vs\to\REAL$ is given by
\begin{align}
  \as_{\pone}(\us,\vs) := 
  \begin{cases}
    \displaystyle\,\int_{\Omega} \nabla\Delta^\ell\us\cdot\nabla\Delta^\ell\vs\,\dx  & \mbox{for~$\pone=2\ell+1$, $\ell\geq0$},\\[1em]
    \displaystyle\,\int_{\Omega} \Delta^\ell\us\,\Delta^\ell\vs\,\dx                 & \mbox{for~$\pone=2\ell$, $\ell\geq1$}.
  \end{cases}
  \label{eq:a:def}
\end{align}
If $\fs\in\LTWO(\Omega)$ we have
\begin{align}
  \BIL{\fs}{\vs} := (\fs,\vs) = \int_{\Omega}\fs\vs\,\dx,
  \label{eq:poly:rhs:def}
\end{align}
where $(\cdot,\cdot)$ denotes the $\LTWO$-inner product.
The bilinear form $\as_{\pone}(\cdot,\cdot)$ is coercive and
continuous with respect to
$\norm{\us}{\Vs}:=(\as_{\pone}(\us,\us))^{1/2}$, which is
a norm on $\HS{\pone}_{0}(\Omega)$.
The coercivity and continuity constants are respectively denoted by
$\alpha$ and $\Ms$, and their value depends on the regularity of
$\Omega$ and its boundary $\Gamma$.
Coercivity and continuity implies existence and uniqueness of the
solution to \eqref{eq:poly:varform} from an application of the
Lax-Milgram theorem~\cite[Theorem~2.7.7]{Brenner-Scott:2008}.
About the regularity of the solution to \eqref{eq:poly:varform}, it is
worth mentioning the result
in~\cite[Corollary~2.21]{Gazzola-Grunau-Sweers:1991}.
Accordingly, if the domain boundary $\partial\Omega$ is
$\CS{k}$-regular for $k\geq 2\pone$ and $\fs\in\HS{k-2\pone}(\Omega)$,
then $\us\in\HS{k}(\Omega)\cap\HS{\pone}_{0}(\Omega)$ and it holds
that $\norm{\us}{k}\leq\Cs\norm{\fs}{k-2\pone}$.
As pointed out in~\cite{Antonietti-Manzini-Scacchi-Verani:2021}, the
regularity of $\us$ for domains with irregular boundaries is still an
open issue.
However, we know that a similar result holds for the biharmonic
problem, i.e., $\pone=2$, if $\Omega$ is a bounded, convex, polygonal
domain see~\cite{Blum-Rannacher:1980}.

\medskip
\section{The discrete problem and an abstract convergence result}
\label{sec3:discrete_pbl}


Let $\rs$ and $\ptwo$ be two integer numbers such that
$\rs\geq\ptwo\geq\pone\geq1$.
The virtual element approximation to the variational
problem~\eqref{eq:poly:varform} reads as
\begin{align}
  \mbox{\emph{Find $\ush\in\Vhrp{r}$ such that:}}\quad
  \ash(\ush,\vsh) = \bil{\fsh}{\vsh}
  \quad\forall \vsh\in\Vhrp{r},
  \label{eq:poly:VEM}
\end{align}
where the virtual element space $\Vhrp{r}$ is a finite-dimensional
conforming subspace of $\Vs$;
$\ash(\cdot,\cdot):\Vhrp{r}\times\Vhrp{r}\to\REAL$ is the virtual
element bilinear form that approximates the bilinear
form~\eqref{eq:a:def}; $\bil{\fsh}{\cdot}:\Vhrp{r}\to\REAL$ is the
continuous linear functional that
approximates~\eqref{eq:poly:rhs:def} through an element $\fsh$ of
the dual space $\Vhrs{r}$ of $\Vhrp{r}$.
The formal definition and properties of $\Vhrp{r}$,
$\ash(\cdot,\cdot)$ and $\fsh$ are discussed in the next section.

\subsection{Mesh notation, mesh regularity and some basic definitions}
The virtual element method is formulated on the mesh family
$\big\{\Th\big\}_{h}$, where each mesh $\Th$ is a partition of the
computational domain $\Omega$ into nonoverlapping polygonal elements
$\P$ and is labeled by the mesh size parameter $\hh$ that is defined
below.
A polygonal element $\P$ is a compact subset of $\REAL^2$ with
boundary $\partial\P$, area $\mP$, center $\xvP$, and diameter
$\hP=\sup_{\xv,\yv\in\P}\vert\xv-\yv\vert$.
The mesh elements of $\Th$ form a finite cover of $\Omega$ such that
$\overline{\Omega}=\cup_{\P\in\Th}\P$ and the mesh size labeling each
mesh $\Th$ is defined by $\hh=\max_{\P\in\Th}\hP$.
A mesh edge $\E$ has center $\xvE$ and length $\hE$ and we denote the
set of mesh edges by $\calE_{\hh}$.
A mesh vertex $\vrtx$ has position vector $\xvV$ and we denote the set
of mesh vertices by $\calV_{\hh}$.
Moreover, in the definition of the degrees of freedom of the next
section, we associate every vertex $\vrtx$ with a characteristic
lenght $\hV$.
This characteristic lenght $\hV$ can be the average of the diameters
of the polygons sharing $\V$.

\medskip
For any integer number $\ell\geq0$, we let $\PS{\ell}(\P)$ and
$\PS{\ell}(\E)$ denote the space of polynomials defined on $\P$ and
$\E$, respectively, and $\PS{\ell}(\Th)$ denotes the space of
piecewise polynomials of degree $\ell$ on the mesh $\Th$.
Accordingly, if $\qs\in\PS{\ell}(\Th)$ then it holds that
$\restrict{\qs}{\P}\in\PS{\ell}(\P)$ for all $\P\in\Th$.
Finally, we define the (broken) seminorm of a function
$\vs\in\prod_{\P\in\Th}\HS{\pone}(\P)$ by
\begin{align*}
  \norm{\vs}{\hh}^2=\sum_{\P\in\Th}\asP(\vs,\vs).
\end{align*}
Throughout the paper, we use the multi-index notation, so that
$\nu=(\nu_1,\nu_2)$ is a two-dimensional index defined by the two
integer numbers $\nu_1,\nu_2\geq0$.
Moreover,
$\Ds^{\nu}\ws=\partial^{|\nu|}\ws\slash{\partial\xs^{\nu_1}\partial\ys^{\nu_2}}$
denotes the partial derivative of order $|\nu|=\nu_1+\nu_2>0$ of a
given bivariate function $\ws(\xs,\ys)$, and we use the conventional
notation that $\Ds^{(0,0)}\ws=\ws$ for $\nu=(0,0)$.
We denote the partial derivatives of $\ws$ versus $x$ and $y$ by
the shortcuts $\partial_{x}\ws$, $\partial_{y}\ws$,
$\partial_{xx}\ws$, $\partial_{xy}\ws$, $\partial_{yy}\ws$, etc.
We denote the normal and tangential derivatives with respect to a
given edge and their mixed combination by $\partial_{n}\ws$,
$\partial_{t}\ws$, $\partial_{tt}\ws$, $\partial_{nt}\ws$,
$\partial_{nn}\ws$, etc, and use the shorter notation
$\partial^{j}_{t}\ws$ and $\partial^{j}_{n}\ws$ for the tangential and
normal derivatives of $\ws$ of order $j$.

\subsection{Abstract convergence theorem}

For the mathematical formulation of the virtual element
approximation~\eqref{eq:poly:VEM}, we require the two following
assumptions on the virtual element space $\Vhrp{r}$ and the bilinear
form $\ash(\cdot,\cdot)$:
\begin{description}
\item \ASSUM{H1}. For all $\hh>0$, the \emph{global} virtual element
  space $\Vhrp{\rs}$ is a conforming, finite-dimensional subspace of
  $\Vs=\HS{\pone}_0(\Omega)\cap\HS{\ptwo}(\Omega)$ such that for all
  elements $\P$ of all mesh partitions $\Th$ it holds that

  \medskip
  \begin{description}
  \item[-] $\VhPrp{\rs}$, the \emph{local} (elemental) virtual element
    space that is defined as the restriction of $\Vhrp{\rs}$ to the
    element $\P$ is a finite-dimensional subspace of $H^{\ptwo}(\P)$;

    \medskip
  \item[-] $\PS{\rs}(\P)$, the space of polynomials of degree up to
    {$r$} defined on $\P$ is a subspace of $\VhPrp{\rs}$.
  \end{description}

  \medskip
\item \ASSUM{H2}. The bilinear form
  $\ash(\cdot,\cdot):\Vhrp{\rs}\times\Vhrp{\rs}\to\REAL$ admits the
  elementwise decomposition
  \begin{align*}
    \ash(\ush,\vsh) = \sum_{\P\in\Th}\ashP(\ush,\vsh)
    \quad\forall \ush,\,\vsh\in\Vhrp{\rs},
  \end{align*}
  where for all element $\P$ the local bilinear form
  $\ashP(\cdot,\cdot)$ is symmetric and such that

  \medskip
  \begin{description}
  \item\textbf{($r$-Consistency)}: for every polynomial
    $\qs\in\PS{\rs}(\P)$ and every virtual element function
    $\vsh\in\Vhrp{\rs}(\P)$ it holds that
    \begin{align}
      \ashP(\vsh,\qs) = \asP(\vsh,\qs);
      \label{eq:poly:r-consistency}
    \end{align}

    \medskip
  \item\textbf{(Stability)}: there exist two positive constants
    $\alpha_*$, $\alpha^*$ independent of $h$ and $\P$ such that for
    every $\vsh\in\VhPrp{\rs}$ it holds that
    \begin{align}
      \alpha_*\asP(\vsh,\vsh)
      \leq\ashP(\vsh,\vsh)\leq
      \alpha^*\asP(\vsh,\vsh).
      \label{eq:poly:stability}
    \end{align}
    The stability constant $\alpha_{*}$ and $\alpha^{*}$ may depend on
    the polynomial approximation degree $\rs$, see, e.g.,
    \cite{Antonietti-Mascotto-Verani:2018} for the case $\pone=1$.
  \end{description}
\end{description}
Assumption \ASSUM{H2} implies that the symmetric bilinear form
$\ash(\cdot,\cdot)$ is coercive and continuous, so that the existence
and uniqueness of the solution $\ush$ follows from an application of
the Lax-Milgram theorem~\cite[Theorem~2.7.7]{Brenner-Scott:2008}.
Under these assumptions we can prove this abstract convergence result.

\medskip
\begin{theorem}
  \label{theorem:poly:abstract:energy:norm}
  Let $\us\in\Vs$ be the solution to the variational
  problem~\eqref{eq:poly:strong} and $\ush\in\Vhrp{\rs}$,
  $\rs\geq\ptwo\geq\pone\geq1$, the solution to the virtual element
  approximation~\eqref{eq:poly:VEM} under
  assumptions~\ASSUM{H1}-\ASSUM{H2}.
  Then, there exists a constant $\Cs$ independent of $\hh$ such that 
  \begin{align}
    \norm{\us-\ush}{\Vs}\leq\Cs
    \Big(
    \norm{\us-\usI}{\Vs} + \norm{\us-\us_{\pi}}{h} + \norm{\fh-\fs}{\Vhrs{\rs}}
    \Big),
    \label{eq:poly:abstract:energy:norm}
  \end{align}
  for every virtual element approximation $\usI$ in $\Vhrp{\rs}$ and
  any piecewise polynomial approximation $\us_{\pi}\in\PS{\rs}(\Th)$
  of $\us$.
  The constant $\Cs$ is proportional to
  $\big(\Ms\slash{\alpha}\big)\,\big(\alpha^{*}\slash{\alpha_{*}}\big)$.
\end{theorem}
The proof of this theorem was first published
in~\cite{Antonietti-Manzini-Verani:2019} for the case with
$\ptwo=\pone$ and then extended to the case for $\ptwo\geq\pone$
in~\cite{Antonietti-Manzini-Scacchi-Verani:2021}.
\section{The virtual element spaces of higher-order continuity}
\label{sec4:VEM}

In this section, we present the formulation of the virtual eleemnt
method of Eq.~\eqref{eq:poly:VEM}.
To this end, we first introduce the local virtual element spaces, the
degrees of freedom, and the global virtual element space, which is
obtained by ``gluing'' in a conforming way the local spaces.
Then, we discuss the computability of the elliptic projection
operator, and the enhancement of the local spaces, which allows us to
compute the orthogonal projection operators onto the subspace of
polynomials of degree up to $\pone-1$.
Finally, we discuss the construction of the bilinear form
$\ash(\cdot,\cdot)$ and the load term $\bil{\fsh}{\cdot}$.

\subsection{Local space definitions}

Let $\P$ be a mesh element.
For $\ptwo\leq\rs\leq 2\ptwo-2$, we consider the local virtual element
space defined as
\begin{multline}
  \label{eq:vem-space-lower}
  \VhPrp{\rs} = \Big\{
  \vsh\in\HS{\ptwo}(\P):\,
  \Delta^{\ptwo}\vsh\in\PS{\rs-2\pone}(\P),\,
  \partial^{\js}_n\vsh\in\PS{\alpha_j(\ptwo,\rs)}(\E),\,\\
  \,\js=0,\ldots,\ptwo-1~\forall\E\in\partial\P
  \Big\},
\end{multline}
where $\alpha_j(\ptwo,\rs)=\max\{2(\ptwo-\js)-1,\rs-\js\}$.
For $\rs\geq 2\ptwo-1$ it holds that $\alpha_j=\rs-\js$ for all
$\js=0,\ldots,\ptwo-1$ and the definition of the local virtual element
space on the element $\P$ takes the simpler form
\begin{multline}
  \label{eq:vem-space-higher}
  \VhPrp{\rs} = \Big
  \{\vsh\in\HS{\ptwo}(\P):\,
  \Delta^{\ptwo}\vsh\in\PS{\rs-2\pone}(\P),\,
  \partial^{\js}_n\vsh\in\PS{\rs-\js}(\E),\,\\
  \js=0,\ldots,\ptwo-1~\forall
  \E\in\partial\P
  \Big\}.
\end{multline}
In both definitions~\eqref{eq:vem-space-lower}
and~\eqref{eq:vem-space-higher} we use the conventional notation that
$\PS{r}(\P)=\{0\}$ if $r<0$.

\begin{remark}
  The space of polynomials $\PS{\rs}(\P)$ is a subspace of
  $\VhPrp{\rs}$ for both definitions~\eqref{eq:vem-space-lower}
  and~\eqref{eq:vem-space-higher}.
\end{remark}

\begin{remark}
  \label{remark:dim:local:space}
  Let $\CARD{\calP}$ denote the cardinality of a (finite dimensional)
  space $\calP$ and $\NPE$ and $\NPV$ the number of edges and vertices
  of element $\P$.
  The dimension of the local virtual element
  space~\eqref{eq:vem-space-lower} is given by
  \begin{align}
    \textrm{dim}\,\VhPrp{\rs}
    &= \CARD{\PS{\rs-2\pone}(\P)}
    + \sum_{\E\in\partial\P}\sum_{j=0}^{\ptwo-1}\CARD{\PS{\alpha_{\js}(\ptwo,\rs)}(\E)}
    - \NPV\frac{(\ptwo+1)\ptwo}{2}
    \nonumber\\[0.4em]
    &= \frac{(\rs-2\pone+1)(\rs-2\pone+2)}{2}
    + \NPE\sum_{j=0}^{\ptwo-1}\big(\alpha_j(\ptwo,\rs)+1\big)
    \nonumber\\[0.4em]
    &\phantom{=}
    - \NPV\frac{(\ptwo+1)\ptwo}{2}.
    \label{eq:Ndofs:low:def}
  \end{align}
  The dimension of the local virtual element
  space~\eqref{eq:vem-space-higher} is given by
  \begin{align}
    \textrm{dim}\,\VhPrp{\rs}
    &= \CARD{\PS{\rs-2\pone}(\P)}
    + \sum_{\E\in\partial\P}\sum_{j=0}^{\ptwo-1}\CARD{\PS{\rs-\js}(\E)}
    - \NPV\frac{(\ptwo+1)\ptwo}{2}
    \nonumber\\[0.4em]
    &= \frac{(\rs-2\pone+1)(\rs-2\pone+2)}{2}
    + \NPE\frac{\ptwo\big(2\rs+3-\ptwo\big)}{2}
    \nonumber\\[0.4em]
    &\phantom{=}
    - \NPV\frac{(\ptwo+1)\ptwo}{2}.
    \label{eq:Ndofs:high:def}
  \end{align}
  
  In both equations~\eqref{eq:Ndofs:low:def}
  and~\eqref{eq:Ndofs:high:def}, the last term of the right-hand side,
  i.e., $\NPV\ptwo(\ptwo+1)/2$, is subtracted to take into account the
  $\CS{\ptwo-1}$-regularity of $\vsh$ at the elemental vertices.
\end{remark}

\subsection{Local degrees of freedom}
\label{subsec:degrees-of-freedom}

Let $\beta_j=\alpha_j-\min\{2(\ptwo-j)-1,\rs-j\}-1$.
For $\rs=2\ptwo-1-k$ with $k=1,\ldots,\ptwo-1$, the virtual element
functions in the elemental space \eqref{eq:vem-space-lower} are
uniquely identified by the following degrees of freedom
\begin{description}
\item\TERM{D1} $\hV^{|\nu|}D^{\nu}\vsh(\V)$, $\ABS{\nu}\leq \ptwo-1$
  for any vertex $\V$ of $\partial\P$;
  
  \medskip
\item\TERM{D2}
  $\displaystyle\hE^{-1+j}\int_{\E}\qs\partial_{n}^j\vsh\dS$ for any
  $\qs\in\PS{\beta_j}(\E)$ and edge $\E$ of $\partial\P$,
  $\js=\ks+1,\ldots,\ptwo-1$;

  \medskip
\item\TERM{D3} $\displaystyle\hP^{-2}\int_{\P}\qs\vsh\dx$ for any
  $\qs\in\PS{r-2\pone}(\P)$.
\end{description}
For $\rs\geq2\ptwo-1$, we note that $\beta_j=\rs-(2\ptwo-\js)$ and we
consider the polynomial edge moments in \TERM{D2} for
$\js=0,\ldots,\ptwo-1$.

\begin{remark}
  \label{rem:ortho:proj:r-2p1}
  The $\LS{2}$-projection operator $\PizP{\rs-2\pone}$ onto the
  polynomial space $\PS{\rs-2\pone}(\P)$ is computable from the
  degrees of freedom \TERM{D3}.
\end{remark}

A virtual element function in $\vsh\in\VhPrp{\rs}$ has the regularity
property that
$\restrict{(\vsh)}{\partial\P}\in\CS{\ptwo-1}(\partial\P)$.
This regularity is reflected by the choice of the degrees of freedom,
and is, indeed, provided by the vertex degrees of freedom of
\TERM{D1}.
Furthermore, the traces of $\vsh$ ($j=0$) and the $j$-th normal
derivatives (up to order $j=\ptwo-1$) on every edge $\E\in\partial\P$
are univariate polynomials of degree at least $\rs-\js$.
The information provided by \TERM{D1} makes it possible to build
polynomial traces of degree higher than $\rs-\js$ if $\rs$ is equal to
$\ptwo$ (or not ``too bigger'' than $\ptwo$ as shown in the following
examples).
In such a case, we can compute the edge traces $\restrict{(\vsh)}{\E}$
and $\restrict{(\partial^j_n\vsh)}{\E}$ by solving the interpolation
problem that uses the vertex values of $\vsh$ and its partial
derivatives.

In the next three examples, we discuss the trace interpolation problem
for $\rs\geq\ptwo$ and $\js=0$, $\js=1$, $\js\geq2$.
This process is also shown in Table~\ref{tab:trace:polynomial:order}
and Fig.~\ref{fig:dofs}

\begin{example}[$\js=0$]
  We derive the higher-order \emph{tangential} derivatives of $\vsh$
  by repetitively applying the differential operator
  $\tv\cdot\nabla=\big(\tsx\partial_x+\tsy\partial_y\big)$ to the
  univariate polynomial trace of $\vsh$, i.e.,
  $\partial^{\ell}_{t}\vsh$ along every elemental edge (recall that
  $\partial^0_t\vsh=\vsh$ for $\ell=0$).
  For example, for $\ell=1,2,3,4$, the tangential derivatives
  $\partial^{\ell}_{t}\vsh$ are given by
  \begin{align*}
    \partial_{t}\vsh(\V_i) &= \tsx\partial_x(\vsh)(\V_i) + \tsy\partial_y(\vsh)(\V_i),\\[0.5em]
    \partial^2_{t}\vsh(\V_i)
    &= \big(\tsx\partial_x+\tsy\partial_y\big)\big(\partial_{t}\vsh\big)(\V_i)\\[0.5em]
    &= \tsx\partial_x\big(\partial_t\vsh\big)(\V_i) + \tsy\partial_y\big(\partial_t\vsh\big)(\V_i)\\[0.5em]
    &= \tsx\tsx\partial_{xx}\vsh(\V_i) + 2\tsx\tsy\partial_{xy}\vsh(\V_i) + \tsy\tsy\partial_{yy}\vsh(\V_i),\\[0.5em]
    \partial^3_{t}\vsh(\V_i)
    &= \big(\tsx\partial_x+\tsy\partial_y\big)\big(\partial^2_{t}\vsh\big)(\V_i)\\[0.5em]
    &= \tsx\partial_x\big(\partial^2_{t}\vsh\big)(\V_i) + \tsy\partial_y\big(\partial^2_{t}\vsh\big)(\V_i)\\[0.5em]
    &= \tsx\tsx\tsx\partial_{xxx}\vsh(\V_i) + 3\tsx\tsx\tsy\partial_{xxy}\vsh(\V_i) + 3\tsx\tsy\tsy\partial_{xyy}\vsh(\V_i)\\[0.5em]
    &\phantom{=} 
    + \tsy\tsy\tsy\partial_{yyy}\vsh(\V_i),\\[0.5em]
    \partial^4_{t}\vsh(\V_i)
    &= \big(\tsx\partial_x+\tsy\partial_y\big)\big(\partial^3_{t}\vsh\big)(\V_i)\\[0.5em]
    &= \tsx\partial_x\big(\partial^3_{t}\vsh\big)(\V_i) + \tsy\partial_y\big(\partial^3_{t}\vsh\big)(\V_i)\\[0.5em]
    &
    =  \tsx\tsx\tsx\tsx\partial_{xxxx}\vsh(\V_i)
    + 4\tsx\tsx\tsx\tsy\partial_{xxxy}\vsh(\V_i)
    + 6\tsx\tsx\tsy\tsy\partial_{xxyy}\vsh(\V_i)\\[0.5em]
    &\phantom{=}
    + 4\tsx\tsy\tsy\tsy\partial_{xyyy}\vsh(\V_i)
    +  \tsy\tsy\tsy\tsy\partial_{yyyy}\vsh(\V_i).
  \end{align*}
  It is easy to recognize the pattern of the combinatorial
  coefficients in these expansions.
  According to the third column ($\js=0$) of
  Table~\ref{table:dofs:D1}, the degrees of freedom \TERM{D1} at
  vertex $\V_i$ for a given regularity index $\ptwo$ yield $\ptwo$
  pieces of information, $\vsh(\V_i)$, $\partial_t\vsh(\V_i)$,
  $\partial^2_t\vsh(\V_i)$, \ldots $\partial^{\ptwo-1}_t\vsh(\V_i)$.
  Since each edge has two vertices, we have $2\ptwo$ pieces of
  information and we can interpolate the edge trace of $\vsh$ as a
  univariate polynomial of degree $2\ptwo-1$.
  Such polynomial degree is clearly bigger than $\rs$ if we choose
  $\rs$ such that $\ptwo\leq\rs<2\ptwo-1$.
  This fact is not in conflict with the property that the virtual
  element space contains the subspace of polynomials of degree $\rs$.
  
  The polynomial degrees of the edge trace of $\vsh$ that we can
  interpolate from the degrees of freedom \TERM{D1}-\TERM{D2} are
  illustrated in Table~\ref{tab:trace:polynomial:order} by the rows
  for $j=0$ and different values of $\ptwo$.
  In this table, the values of $\rs$ such that $\rs=2\ptwo-1$ are
  reported in bold, and the ones for $\rs<2\ptwo-1$ are those
  preceeding the bold ones on the same row.
  For these values of $\rs$ the trace of $\vsh$ can be interpolated
  from the information provided by \TERM{D1}.
  However, if we increase the polynomial degree $\rs$ so that
  $\rs>2\ptwo-1$, the degrees of freedom \TERM{D1} are no longer
  enough to solve the interpolation problem.
  In such a case, we need the additional degrees of freedom of
  \TERM{D2}, i.e., the moments of $\vsh$ against a (basis of)
  polynomials of degree $\rs-(2\ptwo-1)$ defined on $\E$.
  \ENDPROOF
\end{example}

\begin{example}[$j=1$]
  As for the case $j=0$, we derive the higher-order \emph{tangential}
  derivatives of
  $\partial_n\vsh=\nsx\partial_x\vsh+\nsy\partial_y\vsh$ by
  repetitively applying the differential operator
  $\tv\cdot\nabla=\big(\tsx\partial_x+\tsy\partial_y\big)$ to the
  univariate polynomial trace of $\partial_n\vsh$, i.e,
  $\partial^{\ell}_{t}\partial_n\vsh$ along every elemental edge
  (recall that $\partial^0_t\partial_n\vsh=\partial_n\vsh$ for
  $\ell=0$).
  For example, for $\ell=1,2,3$ we find that
  \begin{align*}
    \partial_{t}\partial_n\vsh(\V_i)
    &= \big(\tsx\partial_x+\tsy\partial_y\big)\big(\partial_{n}\vsh\big)(\V_i)\\[0.5em]
    &= \tsx\partial_x\big(\partial_n\vsh\big)(\V_i) + \tsy\partial_y\big(\partial_n\vsh\big)(\V_i)\\[0.5em]
    &= \tsx\nsx\partial_{xx}\vsh(\V_i) + (\tsx\nsy+\tsy\nsx)\partial_{xy}\vsh(\V_i) + \tsy\tsy\partial_{yy}\vsh(\V_i),\\[0.5em]
    \partial^2_{t}\partial_n\vsh(\V_i)
    &= \big(\tsx\partial_x+\tsy\partial_y\big)\big(\partial_{t}\partial_{n}\vsh\big)(\V_i)\\[0.5em]
    &= \tsx\partial_x\big(\partial_{t}\partial_n\vsh\big)(\V_i) + \tsy\partial_y\big(\partial_{t}\partial_n\vsh\big)(\V_i)\\[0.5em]
    &= \tsx\tsx\nsx\partial_{xxx}\vsh(\V_i)
    + \big(\tsx\big(\tsx\nsy+\tsy\nsx\big)+\tsy\big(\tsx\nsy+\tsy\nsx\big)\big)\partial_{xxy}\vsh(\V_i)\\[0.5em]
    &\quad
    + \tsx\tsy\big(\nsx+\nsy\big)\partial_{xyy}\vsh(\V_i)
    + \tsy\tsy\nsy\partial_{yyy}\vsh(\V_i),\\[0.5em]
    \partial^3_{t}\partial_n\vsh(\V_i)
    &= \big(\tsx\partial_x+\tsy\partial_y\big)\big(\partial^2_{t}\partial_{n}\vsh\big)(\V_i)\\[0.5em]
    &= \big(\tsx\partial_x+\tsy\partial_y\big)\big(\partial^2_{n}\vsh\big)(\V_i)\\[0.5em]
    &= \tsx\partial_x\big( \partial^2_{t}\partial_n\vsh(\V_i) \big)(\V_i) + \tsy\partial_y\big( \partial^2_{t}\partial_n\vsh(\V_i) \big)(\V_i),\\[0.5em]
    &= \ldots
  \end{align*}
  According to the fourth column ($\js=1$) of
  Table~\ref{table:dofs:D1}, the degrees of freedom \TERM{D1} at
  vertex $\V_i$ for a given $\ptwo$ yield $\ptwo-1$ pieces of
  information, $\partial_n\vsh(\V_i)$,
  $\partial_t\partial_n\vsh(\V_i)$,
  $\partial^2_t\partial_n\vsh(\V_i)$, \ldots
  $\partial^{\ptwo-2}_t\partial_n\vsh(\V_i)$.
  Since each edge has two vertices, we have $2(\ptwo-1)$ pieces of
  information that we can interpolate as a polynomial of degree
  $2\ptwo-3$.
  Such polynomial degree is clearly bigger than $\rs-1$ if we choose
  $\rs$ such that $\ptwo\leq\rs<2(\ptwo-1)$.

  The polynomial degrees of the edge trace of $\partial_n\vsh$ that we
  can interpolate from the degrees of freedom \TERM{D1}-\TERM{D2} are
  illustrated in Table~\ref{tab:trace:polynomial:order} by the rows
  for $j=1$ and different values of $\ptwo$.
  In this table, the values of $\rs$ such that $\rs-1=2(\ptwo-1)-1$
  are reported in bold, and the ones for $\rs-1<2(\ptwo-1)-1$ are
  those preceeding the bold ones on the same row.
  For these values of $\rs$ the trace of $\partial_n\vsh$ can be
  interpolated from the information provided by \TERM{D1}.
  However, if we increase the polynomial degree $\rs$ so that
  $\rs-1>2(\ptwo-1)-1$, the degrees of freedom \TERM{D1} are no longer
  enough to solve the interpolation problem.
  In such a case, we need the additional degrees of freedom of
  \TERM{D2}, i.e., the moments of $\partial_n\vsh$ against a (basis
  of) polynomials of degree $\rs-2(\ptwo-1)$ defined on $\E$.
  \ENDPROOF
\end{example}

\begin{example}[$j\geq2$]
  As for the cases $j=0$ and $\js=1$, we derive the higher-order
  tangential derivatives of $\partial^j_n\vsh$ by repetitively
  applying the differential operator
  $\tv\cdot\nabla=\big(\tsx\partial_x+\tsy\partial_y\big)$ to the
  univariate polynomial trace of $\partial^j_n\vsh$, i.e.,
  $\partial^{\ell}_{t}\partial^j_n\vsh$ along every elemental edge.
  For example, for $\js=2$, since
  $\partial^2_n\vsh=\nsx\nsx\partial_{xx}\vsh+2\nsx\nsy\partial_{xy}\vsh+\tsy\tsy\partial_{yy}\vsh$,
  we find that
  \begin{align*}
    \partial_{t}\partial^2_n\vsh(\V_i)
    &= \big(\tsx\partial_x+\tsy\partial_y\big)\big(\partial^2_{n}\vsh\big)(\V_i)\\[0.5em]
    &= \tsx\partial_x\big(\partial^2_n\vsh\big)(\V_i) + \tsy\partial_y\big(\partial^2_n\vsh\big)(\V_i)\\[0.5em]
    &= \tsx\nsx\nsx\partial_{xxx}\vsh
    + (2\tsx\tsx\nsx+\tsy\nsx\nsx)\partial_{xxy}\vsh\\[0.5em]
    &+ (\tsx\nsy\nsy+2\tsy\nsx\nsy)\partial_{xyy}\vsh
    + \tsy\nsy\nsy\partial_{yyy}\vsh.
  \end{align*}
  For $\js\geq0$, each edge vertex $\V_i$, $i=1,2$, provides the
  values of $\partial^j_n\vsh$ and its first $\ptwo-\js-1$ tangential
  derivatives.
  Hence, there are $2(\ptwo-\js)$ pieces of information available on
  each edge and we can interpolate the edge trace of
  $\partial^j_n\vsh$ as a univariate polynomial of degree
  $2(\ptwo-\js)-1$.
  Such polynomial degree is clearly bigger than $\rs-\js$ if we choose
  $\rs$ such that $\ptwo\leq\rs<2\ptwo-\js-1$.
  
  The polynomial degrees of the edge trace of $\partial_n^j\vsh$ that
  we can interpolate from the degrees of freedom \TERM{D1}-\TERM{D2}
  are illustrated in Table~\ref{tab:trace:polynomial:order} by the
  rows for $j\geq0$ and different values of $\ptwo$.
  In this table, the values of $\rs$ such that $\rs=2\ptwo-\js-1$ are
  reported in bold, and the ones for $\rs<2\ptwo-\js-1$ are those
  preceeding the bold ones on the same row.
  For these values of $\rs$ the trace of $\partial_n^j\vsh$ can be
  interpolated from the information provided by \TERM{D1}.
  However, if we increase the polynomial degree $\rs$ so that
  $\rs>2\ptwo-\js-1$, the degrees of freedom \TERM{D1} are no longer
  enough to solve the interpolation problem.
  In such a case, we need the additional degrees of freedom of
  \TERM{D2}, i.e., the moments of $\partial_n^j\vsh$ against a (basis
  of) polynomials of degree $\rs-(2\ptwo-\js-1)$ defined on $\E$.
  It is worth noting that for increasing values of $\rs$, we need to
  supplement this information starting from the higher-order normal
  derivatives (see Fig.~\ref{fig:dofs}).
  \ENDPROOF
\end{example}


\newcommand{\hspc}{\hspace{1.25mm}}
\renewcommand{\arraystretch}{1.75}
\renewcommand{\TABROW} [5]{ \hspc#1\hspc\hspc & \hspc#2\hspc\hspc & \hspc#3\hspc\hspc & \hspc#4\hspc\hspc & \hspc#5\hspc\\}
\newcommand{\TABMROW}[6]{ \multirow{#1}{*}{\hspc#2\hspc\hspc} & \multirow{#1}{*}{\hspc#3\hspc\hspc} & \multirow{#1}{*}{\hspc#4\hspc\hspc} & \multirow{#1}{*}{\hspc#5\hspc\hspc} & \multirow{#1}{*}{\hspc#6\hspc}\\}
\begin{table}[t]
  \centering
  \begin{tabular}{c|c|c|c|c|c|c}
    \hspc$\ptwo-1$\hspc\hspc & \hspc Degrees of freedom \TERM{D1}\hspc & \TABROW{$\js=0$}{$\js=1$}{$\js=2$}{$\js=3$}{$\js=4$}
    \hline\hline
    $0$ &
    $\vsh(\V_i)$ &
    \TABROW{$\vsh$}{$--$}{$--$}{$--$}{$--$}
    \hline
    $1$ &
    $\partial_x\vsh(\V_i),\partial_y\vsh(\V_i)$ &
    \TABROW{$\partial_t\vsh$}{$\partial_n\vsh$}{$--$}{$--$}{$--$}
    \hline
    \multirow{2}{*}{$2$} &
    \multirow{2}{*}{\shortstack[1]{$\partial_{xx}\vsh(\V_i),\partial_{xy}\vsh(\V_i)$,\\ $\partial_{yy}\vsh(\V_i)$}} &
    \TABMROW {2}{$\partial^2_t\vsh$}{$\partial_t\partial_n\vsh$}{$\partial^2_n\vsh$}{$--$}{$--$}
    {}  & {} & \TABROW{}{}{}{}{}
    \hline
    \multirow{2}{*}{$3$} &
    \multirow{2}{*}{\shortstack[1]{$\partial_{xxx}\vsh(\V_i),\partial_{xxy}\vsh(\V_i)$,\\  $\partial_{xyy}\vsh(\V_i),\partial_{yyy}\vsh(\V_i)$}} &
    \TABMROW {2}{$\partial^3_t\vsh$}{$\partial^2_t\partial_n\vsh$}{$\partial_t\partial^2_n\vsh$}{$\partial^3_n\vsh$}{$--$}
    {}  & {} & \TABROW{}{}{}{}{}
    \hline
    \multirow{3}{*}{$4$} & 
    \multirow{3}{*}{\shortstack[1]{$\partial_{xxxx}\vsh(\V_i),\partial_{xxxy}\vsh(\V_i)$,\\ $\partial_{xxyy}\vsh(\V_i),\partial_{xyyy}\vsh(\V_i)$, \\ $\partial_{yyyy}\vsh(\V_i)$}} &
    \TABMROW {3}{$\partial^4_t\vsh$}{$\partial^3_t\partial_n\vsh$}{$\partial^2_t\partial^2_n\vsh$}{$\partial_t\partial^3_n\vsh$}{$\partial^4_n\vsh$}
    {}  & {} & \TABROW{}{}{}{}{}
    {}  & {} & \TABROW{}{}{}{}{}
    \hline
  \end{tabular}
  \vspace{0.5cm}
  \caption{Vertex degrees of freedom for the trace interpolation
    process on the elemental edges.
    The first column shows the value of $\max\{\ABS{\nu}\}=\ptwo-1$
    that we use to define the degrees of freedom \TERM{D1}.
    The second column shows the degrees of freedom at vertex $\V_i$
    corresponding to $\ptwo-1$ in the first column.
    The remaining columns shows the quantities that we can compute
    using the degrees of freedom listed in the second column.
    Recalling that $\ABS{\nu}\leq\ptwo-1$ and $0\leq\js\leq\ptwo-1$,
    on the columns for $\js=0,\ldots,4$ we read the pieces of
    information that are available for the interpolation of
    $\partial^j_n\vsh$ (with $\partial^0_n\vsh=\vsh$ for $\js=0$).
    For example, if $\ptwo=3$, we can use only the objects of the
    first three table rows (i.e., $\ABS{\nu}=0,1,2$), and each column
    for $\js=0,1,2$ lists the pieces of information that are available
    to construct the polynomial trace of $\vsh$, $\partial_n\vsh$ and
    $\partial^2_n\vsh$ on each edge.
    In such a case, the vertex degrees of freedom allows us to
    interpolate the trace of $\vsh$ as a polynomial of degree $5$, the
    trace of $\partial_n\vsh$ as a polynomial of degree $3$, and the
    trace of $\partial^2_n\vsh$ as a polynomial of degree $1$.
    These trace interpolations are consistent with $\rs=\ptwo$ and a
    global virtual element space with $\CS{\ptwo-1}$-regularity on
    $\Omega$. }
  \label{table:dofs:D1}
\end{table}
\renewcommand{\arraystretch}{1}
\renewcommand{\TABROW}[9]{
  $#1$\,\, & \,\,$#2$\,\, & \,\,$#3$\,\, & \,\,$#4$\,\, & \,\,$#5$\,\, & \,\,$#6$\,\, & \,\,$#7$\,\, & \,\,$#8$\,\, & \,\,$#9$ &\ldots\\
}

\newcommand{\TABRBF}[9]{
  $\mathbf{#1}$\,\, & \,\,$\mathbf{#2}$\,\, & \,\,$\mathbf{#3}$\,\, & \,\,$\mathbf{#4}$\,\, & \,\,$\mathbf{#5}$\,\, & \,\,$\mathbf{#6}$\,\, & \,\,$\mathbf{#7}$\,\, & \,\,$\mathbf{#8}$\,\, & \,\,$\mathbf{#9}$ &\ldots\\
}

\begin{table}[t]
  \centering
  \renewcommand{\arraystretch}{1.25}
  \begin{tabular}{c||c||c|c|c|c|c|c|cl}
    \hline
        {} & {} & \multicolumn{7}{c}{$\rs-j$} & \\
        \hline
        \TABRBF{\ptwo=1}{2(\ptwo-j)-1}{\rs=1}{\rs=2}{\rs=3}{\rs=4}{\rs=5}{\rs=6}{\rs=7}
        \TABROW{j=0}{\mathbf{1}}{\mathbf{1}}{2}{3}{4}{5}{6}{7}
        \hline
        \TABRBF{\ptwo=2}{2(\ptwo-j)-1}{\rs=1}{\rs=2}{\rs=3}{\rs=4}{\rs=5}{\rs=6}{\rs=7}
        \TABROW{j=0}{\mathbf{3}}{-}{2}{\mathbf{3}}{4}{5}{6}{7}
        \TABROW{j=1}{\mathbf{1}}{-}{\mathbf{1}}{2}{3}{4}{5}{6}
        \hline
        \TABRBF{\ptwo=3}{2(\ptwo-j)-1}{\rs=1}{\rs=2}{\rs=3}{\rs=4}{\rs=5}{\rs=6}{\rs=7}
        \TABROW{j=0}{\mathbf{5}}{-}{-}{3}{4}{\mathbf{5}}{6}{7}
        \TABROW{j=1}{\mathbf{3}}{-}{-}{2}{\mathbf{3}}{4}{5}{6}
        \TABROW{j=2}{\mathbf{1}}{-}{-}{\mathbf{1}}{2}{3}{4}{5}
        \hline
        \TABRBF{\ptwo=4}{2(\ptwo-j)-1}{\rs=1}{\rs=2}{\rs=3}{\rs=4}{\rs=5}{\rs=6}{\rs=7}
        \TABROW{j=0}{\mathbf{7}}{-}{-}{-}{4}{5}{6}{\mathbf{7}}
        \TABROW{j=1}{\mathbf{5}}{-}{-}{-}{3}{4}{\mathbf{5}}{6}
        \TABROW{j=2}{\mathbf{3}}{-}{-}{-}{2}{\mathbf{3}}{4}{5}
        \TABROW{j=3}{\mathbf{1}}{-}{-}{-}{\mathbf{1}}{2}{3}{4}
        \hline
  \end{tabular}
  \renewcommand{\arraystretch}{1}
  \caption{Polynomial orders of the edge traces of $\vsh$ and its
    normal derivatives $\partial^j_n\vsh$ for $\ptwo=1,2,3,4$ and
    $\rs=\ptwo,\ldots7$ (we recall that $\rs\geq\ptwo$).
    The first column on the left reports the value of $\ptwo$ and
    $\js=0,\ldots,\ptwo-1$; the second column reports the value of
    $2(\ptwo-\js)-1$, which is a threshold value, and the remaining
    columns the possible values of $\rs-\js$ (remember that on each
    edge $\partial^{\js}_n\vsh\in\PS{\alpha_j}$ with
    $\alpha_j=\max\{2(\ptwo-\js)-1,\rs-\js\}$).
    The values of the polynomial degree $\rs-1$ such that
    $\rs-\js=2(\ptwo-\js)-1$ (or, equivalently, that
    $\rs=2\ptwo-\js-1$) are reported in bold font.
    The polynomial traces with degree equal or higher than this bold
    value, which are above it in every column ad correspond to the
    smaller order $\js$ of derivation, can be interpolated using only
    the vertex degrees of freedom \TERM{D1}.
    To interpolate the remaining edge traces we need the additional
    information provided by \TERM{D2}.  }
  \label{tab:trace:polynomial:order}
\end{table}

\begin{figure}[t]
  \begin{center}
    \begin{tabular}{cccc}
      \includegraphics[scale=0.25]{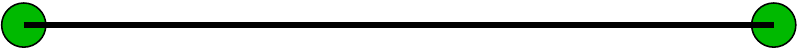} 
      & \hspace{0.5cm}
      \includegraphics[scale=0.25]{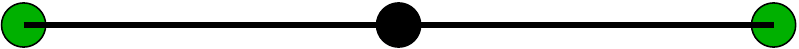} 
      & \hspace{0.5cm}
      \includegraphics[scale=0.25]{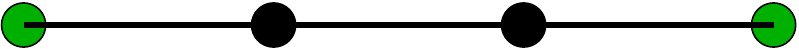} 
      & \hspace{0.5cm}
      \includegraphics[scale=0.25]{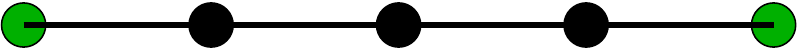} 
      \\
      $\ptwo=1,\,\rs=1$ & \hspace{0.5cm} $\ptwo=1,\,\rs=2$
      & \hspace{0.5cm} $\ptwo=1,\,\rs=3$ & \hspace{0.5cm}
      $\ptwo=1,\,\rs=4$ \\[0.75em]
      \includegraphics[scale=0.25]{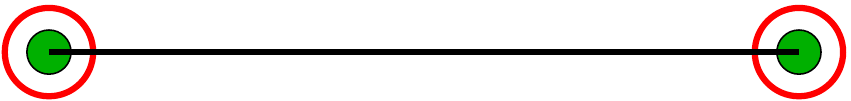} 
      & \hspace{0.5cm}
      \includegraphics[scale=0.25]{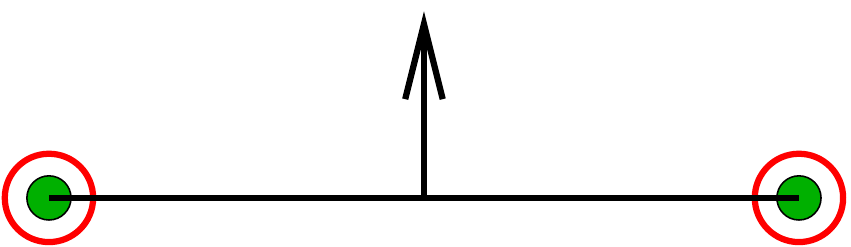} 
      & \hspace{0.5cm}
      \includegraphics[scale=0.25]{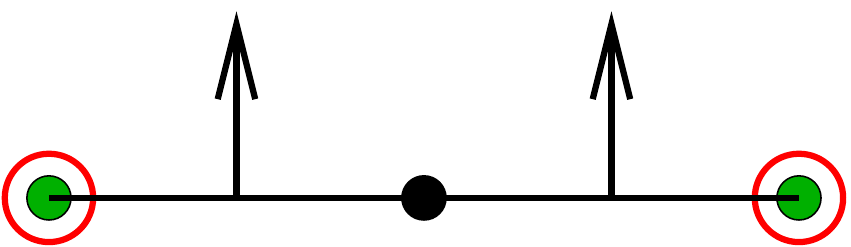} 
      & \hspace{0.5cm}
      \includegraphics[scale=0.25]{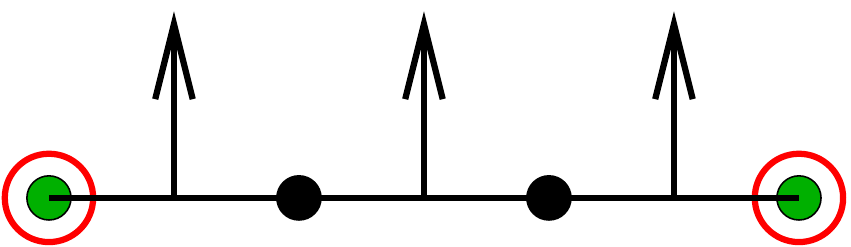} 
      \\
      $\ptwo=2,\,\rs=2$ & \hspace{0.5cm} $\ptwo=2,\,\rs=3$
      & \hspace{0.5cm} $\ptwo=2,\,\rs=4$ & \hspace{0.5cm} $\ptwo=2,\,\rs=5$ \\[0.75em]
      \includegraphics[scale=0.25]{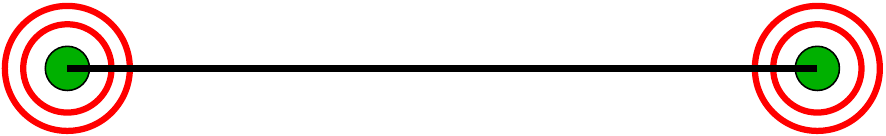} 
      & \hspace{0.5cm}
      \includegraphics[scale=0.25]{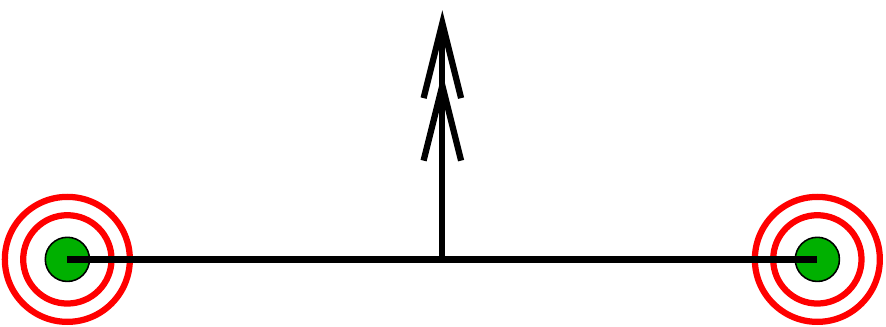} 
      & \hspace{0.5cm}
      \includegraphics[scale=0.25]{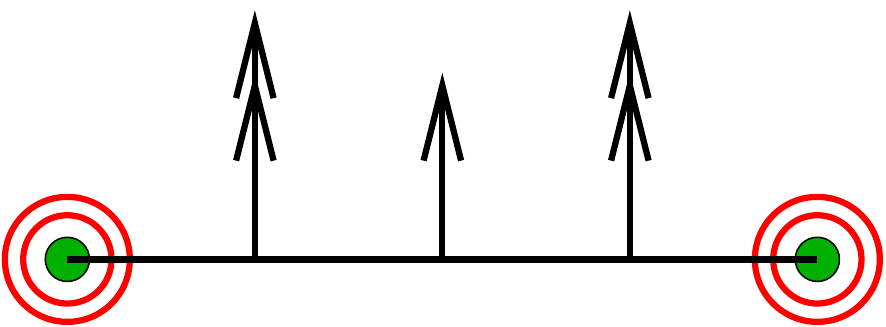} 
      & \hspace{0.5cm}
      \includegraphics[scale=0.25]{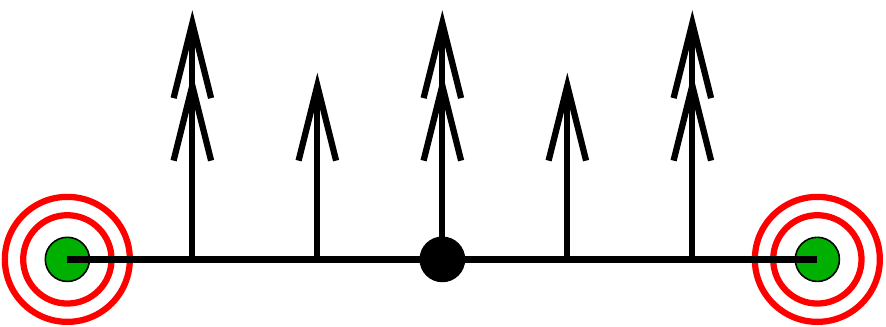} 
      \\
      $\ptwo=3,\,\rs=3$ & \hspace{0.5cm} $\ptwo=3,\,\rs=4$ & \hspace{0.5cm} $\ptwo=3,\,\rs=5$ & \hspace{0.5cm} $\ptwo=3,\,\rs=6$ \\[0.75em]
      \includegraphics[scale=0.25]{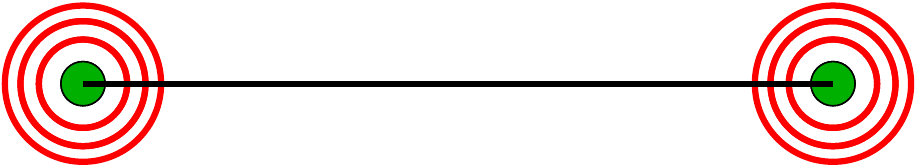} 
      & \hspace{0.5cm}
      \includegraphics[scale=0.25]{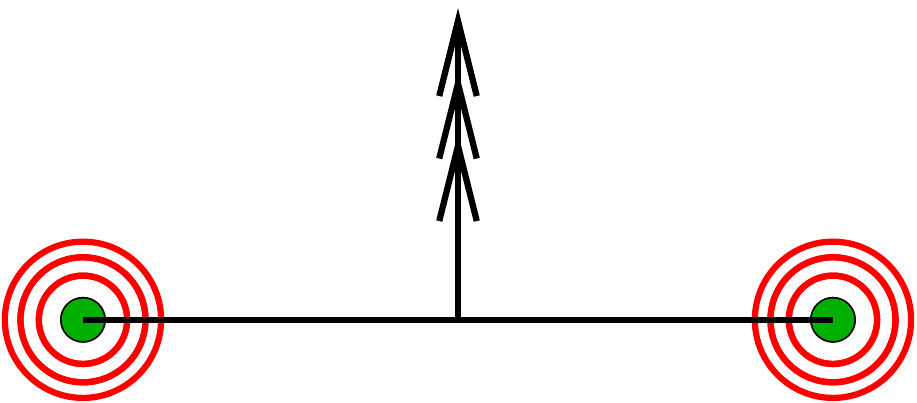} 
      & \hspace{0.5cm}
      \includegraphics[scale=0.25]{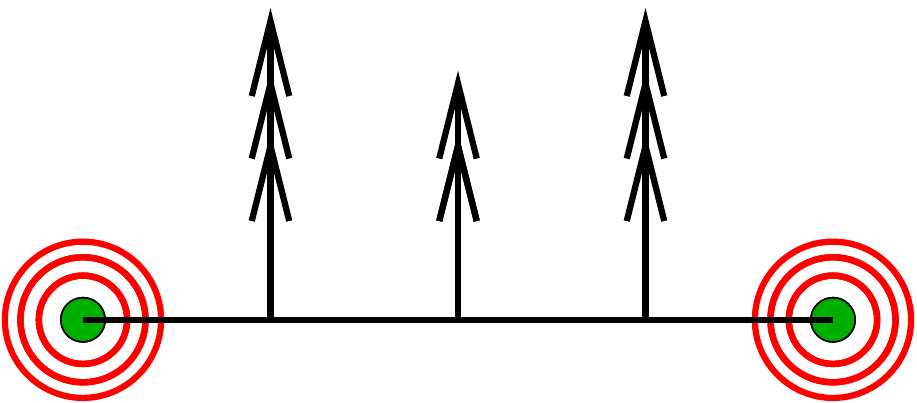} 
      & \hspace{0.5cm}
      \includegraphics[scale=0.25]{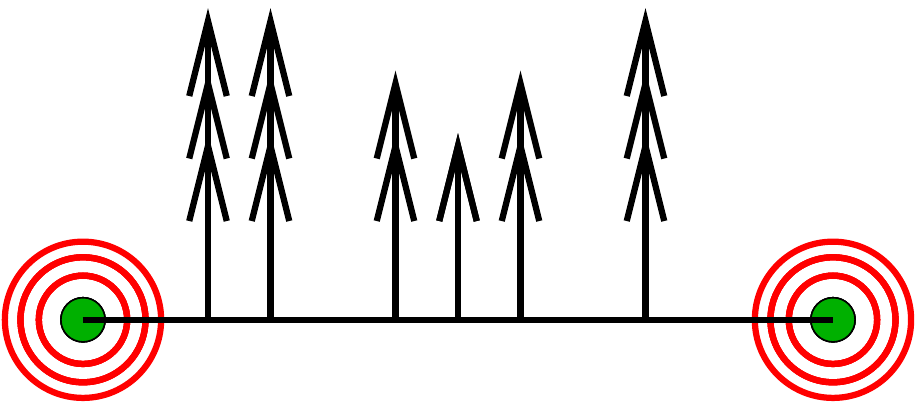} 
      \\
      $\ptwo=4,\,\rs=4$ & \hspace{0.5cm} $\ptwo=4,\,\rs=5$ & \hspace{0.5cm} $\ptwo=4,\,\rs=6$ & \hspace{0.5cm} $\ptwo=4,\,\rs=7$ \\[0.75em]
    \end{tabular}
    \caption{Edge degrees of freedom of the virtual element space
      $\Vhrp{\rs}$ with regularity index $\ptwo=1$ (Laplace operator),
      $\ptwo=2$ (bi-harmonic operator), $\ptwo=3$ (tri-harmonic
      operator), $\ptwo=4$, and polynomial degree $\rs$ such that
      $\ptwo\leq\rs\leq\ptwo+4$.
      The (green) dots at the vertices represent the vertex values and
      each (red) vertex circle represents an order of derivation.
      The (black) dots on the edge represent the polynomial moments of
      the trace $\restrict{\vsh}{\E}$; the arrows represent the
      polynomial moments of $\restrict{\partial_n\vsh}{\E}$; the
      double arrows represent the polynomial moments of
      $\restrict{\partial^2_{n}\vsh}{\E}$.  }
    \label{fig:dofs}
  \end{center}
\end{figure}

\begin{lemma}
  \label{lemma:unisolvence}
  The degrees of freedom \TERM{D1}-\TERM{D3} are unisolvent in the
  virtual element space $\VhPrp{r}$.
\end{lemma}
\begin{proof}
  Let $\P$ be a polygonal element.
  First, a counting argument shows that the number of degrees of
  freedom \TERM{D1}-\TERM{D3} is equal to the dimension of
  $\VhPrp{\rs}$ (see Remark~\ref{remark:dim:local:space}).
  Then, we prove that a virtual element function $\vsh$ is necessarily
  zero if all its degrees of freedom \TERM{D1}-\TERM{D3} are zero.
  In particular, assuming that the degrees of freedom \TERM{D1} and
  \TERM{D2} are zero implies that the polynomial traces of $\vsh$ and
  its normal derivatives of order up to $\ptwo-1$ are identically zero
  on all edges of $\partial\P$, and so are their tangential
  derivatives of any order.
  Likewise, assuming that the degrees of freedom \TERM{D3} are zero
  implies that the elemental moments of $\vsh$ against the polynomials
  of degree up to $\rs-2\pone$ (for $\rs\geq2\pone$) are zero.
  
  Consider separately the case of odd and even values of $\ptwo$.
  If $\ptwo=2\ell+1$ with $\ell\geq0$, a repeated application of the
  integration by parts formula yields
  \begin{align}
    \int_{\P}\ABS{\nabla\Delta^{\ell}\vsh}^2\,\dx =
    &
    - \int_{\P} \big(\Delta^{\ptwo}\vsh\big)\,\vsh\,\dx
    + \int_{\partial\P}\big(\partial_n\Delta^\ell\vsh\big)\,\Delta^\ell\vsh\dS\nonumber \\[0.5em]
    & +\sum_{i=1}^\ell
    \left(
    \int_{\partial\P}\big(\partial_n\Delta^{\ptwo-i}\vsh\big)\,\Delta^{i-1}\vsh\dS -
    \int_{\partial\P}\big(\Delta^{\ptwo-i}\vsh\big)\,\partial_n\Delta^{i-1}\vsh\dS
    \right).
    \label{eq:poly:intbyparts:odd:p}
  \end{align}
  Similarly, if $\ptwo=2\ell$ with $\ell\geq1$, we find that 
  \begin{align}
    \int_{\P}\ABS{\Delta^{\ell}\vsh}^2\,\dx
    &= \int_{\P}\big(\Delta^{\ptwo}\vsh\big)\,\vsh\,\dx
    \nonumber\\[0.5em]
    &\phantom{=}
    -\sum_{i=1}^\ell
    \left(
    \int_{\partial\P}  \big(\partial_n\Delta^{\ptwo-i}\vsh\big)\,\Delta^{i-1}\vsh\,\dS
    -\int_{\partial\P} \big(\Delta^{\ptwo-i}\vsh\big)\,\partial_n\Delta^{i-1}\vsh\,\dS
    \right).
    \label{eq:poly:intbyparts:even:p}
  \end{align}
  
  Since $\Delta^{\ptwo}\vsh$ is a polynomial of degree $\rs-2\pone$
  according to the definition of the virtual element space, the volume
  integral in the right-hand sides of~\eqref{eq:poly:intbyparts:odd:p}
  and~\eqref{eq:poly:intbyparts:even:p} is an elemental moment of
  $\vsh$.
  This integral must be zero since we assumed that the degrees of
  freedom \TERM{D3} of $\vsh$ are zero.

  To prove that the edge integrals in~\eqref{eq:poly:intbyparts:odd:p}
  and~\eqref{eq:poly:intbyparts:even:p} are zero, we first note that
  such integrals contain the edge trace of $\Delta^{\mu}\vsh$ for
  $\mu=0,\ldots,\ptwo-1$ and its normal and tangential derivatives.
  Since $\restrict{\Delta\vsh}{\E}=\partial^2_t\vsh+\partial^2_n\vsh$,
  it holds that 
  \begin{align}
    \Delta^{\mu}\vsh
    = \big(\partial^2_t+\partial^2_n\big)^\mu\vsh
    = \sum_{\nu=0}^{\mu}\Cs_{\mu,\nu}\,\partial^{2(\mu-\nu)}_t\,\partial^{2\nu}_n\vsh,
    \label{eq:proof:unisolvence:00}
  \end{align}
  where $\Cs_{\mu,\nu}$ denote the $\nu$-th combinatorial coefficient
  of the $\mu$-th power expansion.
  Therefore, all the edge integrals either contain the normal
  derivatives $\partial^{\ell}_n\vsh$ for some integer
  $\ell=0,\ldots,\ptwo-1$, or the tangential derivatives of these
  quantities.
  As noted at the beginning of this proof, all these quantities are
  zero since we assumed that the degrees of freedom
  \TERM{D1}-\TERM{D2} of $\vsh$ are zero.

  Finally, we note that a function $\vsh\in\VhPrp{\rs}$ with all zero
  degrees of freedom also belongs to
  $\HS{\ptwo}_{0}(\P)=\{\vs\in\HS{\ptwo}(\P):\,\restrict{\partial^j\vs}{\partial\P}=0\,\forall\js=0,\ldots,\ptwo-1\}$.
  Since both left-hand sides of~\eqref{eq:poly:intbyparts:odd:p}
  and~\eqref{eq:poly:intbyparts:even:p} are a norm on
  $\HS{\ptwo}_{0}(\P)$, it follows that $\vsh=0$.
  \ENDPROOF
\end{proof}

\subsection{Global virtual element spaces}

Building upon the local virtual element spaces, the \emph{global}
conforming virtual element space $\Vhrp{r}$ is defined on $\Omega$ as
\begin{align}
  \Vhrp{\rs} = \Big\{
  \vsh\in\HS{\pone}_{0}(\Omega)\cap\HS{\ptwo}(\Omega)\,:\,\restrict{\vsh}{\P}\in\VhPrp{\rs}\,\,\forall\P\in\Th
  \Big\},
  \label{eq:poly:global:space}
\end{align}
where $\VhPrp{\rs}$ is the local space defined
in~\eqref{eq:vem-space-lower} if $\ptwo\leq\rs\leq2\ptwo-2$ and the
local space defined in~\eqref{eq:vem-space-higher} if
$\rs\geq2\ptwo-1$.

\begin{remark}
  \label{remark:dim:global:space}
  Let $\NP$, $\NE$ and $\NV$ denote the
  number of element, edges and vertices of $\Th$.
  The dimension of the global virtual element space built
  upon~\eqref{eq:vem-space-lower} is given by
  \begin{align*}
    \textrm{dim}\,\Vhrp{\rs}
    &
    = \NP\frac{(\rs-2\pone+1)(\rs-2\pone+2)}{2}
    + \NE\sum_{j=0}^{\ptwo-1}\big(\alpha_j(\ptwo,\rs)+1\big)
    \\[0.4em]
    &\phantom{=}
    - \NV\frac{(\ptwo+1)\ptwo}{2}.
  \end{align*}
  The dimension of the global virtual element space built
  upon~\eqref{eq:vem-space-higher} is given by
  \begin{align*}
    \textrm{dim}\,\Vhrp{\rs}
    &
    = \NP\frac{(\rs-2\pone+1)(\rs-2\pone+2)}{2}
    + \NE\frac{(\ptwo-1)\big(2(\rs+1)-(\ptwo-1)\big)}{2}
    \\[0.4em]
    &\phantom{=}
    - \NV\frac{(\ptwo+1)\ptwo}{2}.
  \end{align*}
\end{remark}

The set of global degrees of freedom are inherited from the local
degrees of freedom of section~\ref{subsec:degrees-of-freedom}.
Therefore, we consider
\begin{description}
\item\TERM{D1} $\hV^{|\nu|}D^{\nu}\vsh(\V)$, $\ABS{\nu}\leq \ptwo-1$
  for every vertex $\V$ of $\calV_{\hh}$;
  
  \medskip
\item\TERM{D2}
  $\displaystyle\hE^{-1+j}\int_{\E}\qs\partial_{n}^j\vsh\dS$ for any
  $\qs\in\PS{\beta_j}(\E)$ and every edge $\E$ of $\Eh$,
  $\js=\ks+1,\ldots,\ptwo-1$;

  \medskip
\item\TERM{D3} $\displaystyle\hP^{-2}\int_{\P}\qs\vsh\dx$ for any
  $\qs\in\PS{r-2\pone}(\P)$ and every element $\P$ of $\Th$,
\end{description}
where, again, $\beta_j=\alpha_j-\min\{2(\ptwo-j)-1,\rs-j\}-1$ and
$\alpha_j(\ptwo,\rs)=\max\{2(\ptwo-\js)-1,\rs-\js\}$, $\js=0,\ldots,\ptwo-1$.
For $\rs\geq2\ptwo-1$, these degrees of freedom become

\medskip
\begin{description}
\item\TERM{D1} $\hV^{|\nu|}\Ds^{\nu}\vsh(\V)$, $\ABS{\nu}\leq \ptwo-1$ for every
  interior vertex $\V$ of $\calV_{\hh}$;

  \medskip
\item\TERM{D2} $\displaystyle\hE^{-1+j}\int_{\E}\qs\partial^j_n\vsh\,ds$ for
  any $\qs\in\PS{r-2\ptwo+j}(e)$ $j=0,\ldots,\ptwo-1$ and every interior
  edge $\E$ of $\Eh$;

  \medskip
\item\TERM{D3} $\displaystyle\hP^{-2}\int_{\P}\qs\vsh\dx$ for any
  $\qs\in\PS{r-2\pone}(\P)$ and every element $\P$ of $\Th$.
\end{description}
We remark that the associated global space is made of
$\HS{\ptwo}(\Omega)$ functions.
Indeed, the restriction of a virtual element function $\vsh$ to each
element $\P$ belongs to $\HS{\ptwo}(\P)$ and glues with
$C^{\ptwo-1}$-regularity across the internal mesh faces.

Finally, the unisolvence of these degrees of freedom is an immediate
consequence of the unisolvence of the elementwise degrees of freedom
in any elemental space $\VhPrp{\rs}$,
cf. Lemma~\ref{lemma:unisolvence}.

\subsection{Elliptic projection operator}

The elliptic projection operator
$\PiPr{\rs}:\VhPrp{\rs}\to\PS{\rs}(\P)$ is such that for all
$\vsh\in\VhPrp{\rs}$, the projection $\PiPr{r}\vsh$ is the solution of
the finite dimensional variational problem
\begin{align}
  \asP(\PiPr{r}\vsh-\vsh,\qs)                     &=0 \quad\forall\qs\in\PS{\rs}(\P),\label{eq:poly:Pi:A}\\[0.5em]
  \int_{\partial\P}\big(\PiPr{r}\vsh-\vsh\big)\qs\dS &=0 \quad\forall\qs\in\PS{\pone-1}(\P).\label{eq:poly:Pi:B}  
\end{align}
Condition~\eqref{eq:poly:Pi:B} allows us to fix the nontrivial kernel
of $\asP(\cdot,\cdot)$, which is the subspace of polynomials of degree
(up to) $\pone-1$.

\medskip
\begin{remark}
  Instead of~\eqref{eq:poly:Pi:B}, we can consider the alternative
  condition~\cite{Antonietti-Manzini-Verani:2019}
  \begin{align*}
    \widehat{\Pi}^{\P}\Ds^{\nu}\PiPr{r}\vsh
    &= \widehat{\Pi}^{\P}\Ds^{\nu}\vsh\quad\textrm{with~}\ABS{\nu}\leq \pone-1,
  \end{align*}
  by using the \emph{vertex average projection}
  $\widehat{\Pi}^{\P}:\CS{}(\P)\to\PS{0}(\P)$, which is such that
  \begin{align}
    \widehat{\Pi}^{\P}\psi = \frac{1}{\NP}\sum_{\vrtx\in\partial\P}\psi({\vrtx}),
  \label{eq:trih:vertex:average:projection}
  \end{align}
  for all continuous function $\psi$.
\end{remark}

\begin{lemma}
  \label{lemma:elliptic:projection}
  The elliptic projection operator $\PiPr{\rs}$ is polynomial
  preserving in the sense that $\PiPr{\rs}\qs=\qs$ for every
  $\qs\in\PS{\rs}(\P)$.
\end{lemma}
\begin{proof}
  Let $\PS{\rs}(\P)\setminus\PS{\pone-1}(\P)$ denote the linear space
  of polynomials of degrees $\ss$ such that $\pone\leq\ss\leq\rs$, and
  consider the decomposition
  \begin{align}
    \PS{\rs}(\P)=\PS{\rs}(\P)\setminus\PS{\pone-1}(\P)\oplus\PS{\pone-1}(\P).
  \end{align}
  We expand the polynomial $\qs\in\PS{\rs}(\P)$ and its projection
  $\PiPr{\rs}\qs$ as follows
  \begin{align}
    \qs           &= \sum_{\ell'}\cs_{\ell'}(\qs)\mu_{\ell'}           + \sum_{\ell'}\cshat_{\ell'}(\qs)\muh_{\ell'},\label{eq:qs:expansion}\\[0.5em]
    \PiPr{\rs}\qs &= \sum_{\ell'}\cs_{\ell'}(\PiPr{\rs}\qs)\mu_{\ell'} + \sum_{\ell'}\cshat(\PiPr{\rs}\qs)\muh_{\ell'},\label{eq:PiPr:qs:expansion}
  \end{align}
  where $\{\mu_{\ell'}\}$ is a basis of
  $\PS{\rs}(\P)\setminus\PS{\pone-1}(\P)$, $\{\muh_{\ell'}\}$ is a
  basis of $\PS{\pone-1}(\P)$, and $\cs_{\ell'}(\qs)$,
  $\cshat_{\ell'}(\qs)$, $\cs_{\ell'}(\PiPr{\rs}\qs)$, and
  $\cshat(\PiPr{\rs}\qs)$ are the coefficients of such expansions.
  The range of the summation index $\ell'$, which is not expicitly
  indicated in~\eqref{eq:qs:expansion}
  and~\eqref{eq:PiPr:qs:expansion}, is consistent with the dimensions
  of $\PS{\rs}(\P)\setminus\PS{\pone-1}(\P)$ and $\PS{\pone-1}(\P)$.
  We assume that the polynomials $\mu_{\ell'}$ are orthogonal with
  respect to the semi-inner product $\asP(\cdot,\cdot)$, which is the
  restriction of $\asP(\cdot,\cdot)$ to a polygonal element $\P$, so
  that $\asP(\mu_{\ell'},\mu_{\ell})=\mP\delta_{\ell',\ell}$.
  Since the polynomials $\muh_{\ell'}$ belong to the kernel of
  $\asP(\cdot,\cdot)$, we substitute the
  expansions~\eqref{eq:qs:expansion} and~\eqref{eq:PiPr:qs:expansion}
  in~\eqref{eq:poly:Pi:A} (with $\vsh=\qs$ and $\qs=\mu_{\ell}$) and
  we find that
  \begin{align*}
    0
    &= \asP(\PiPr{\rs}\qs-\qs,\mu_{\ell})
    = \sum_{\ell'}\big(\cs_{\ell'}(\PiPr{\rs}\qs)-\cs_{\ell'}(\qs)\big)\asP(\mu_{\ell},\mu_{\ell'})\\[0.5em]
    &= \mP\sum_{\ell'}\big(\cs_{\ell'}(\PiPr{\rs}\qs)-\cs_{\ell'}(\qs)\big)\delta_{\ell,\ell'}
    = \mP\big(\cs_{\ell}(\PiPr{\rs}\qs)-\cs_{\ell}(\qs)\big),
  \end{align*}
  which holds for all possible integers $\ell$.
  This relation implies that
  \begin{align}
    \PiPr{\rs}\qs-\qs
    = \sum_{\ell'}\big(\cshat_{\ell'}(\PiPr{\rs}\qs)-\cshat_{\ell'}(\qs)\big)\muh_{\ell'}\in\PS{\pone-1}(\P).
    \label{eq:PiPr:diff}
  \end{align}
  Then, we assume that the polynomials $\muh_{\ell'}$ are orthogonal
  with respect to the inner product
  $(\vs,\us)_{\partial\P}=\int_{\partial\P}\vs\us\dS$, so that
  $(\muh_{\ell'},\muh_{\ell})_{\partial\P}=\ABS{\partial\P}\delta_{\ell',\ell}$,
  where $\ABS{\partial\P}$ is the perymeter of $\partial\P$.
  We substitute~\eqref{eq:PiPr:diff} in~\eqref{eq:poly:Pi:B} (with
  $\vsh=\qs$ and $\qs=\muh_{\ell}$) and we find that
  \begin{align*}
    0
    &= \int_{\partial\P}\big(\PiPr{r}\qs-\qs\big)\muh_{\ell}\dS
    =  \sum_{\ell'}\big(\cshat_{\ell'}(\PiPr{\rs}\qs)-\cshat_{\ell'}(\qs)\big)\int_{\partial\P}\muh_{\ell'}\muh_{\ell}\dS\\[0.5em]
    &= \ABS{\partial\P}\sum_{\ell'}\big(\cshat_{\ell'}(\PiPr{\rs}\qs)-\cshat_{\ell'}(\qs)\big)\delta_{\ell',\ell}
    =  \ABS{\partial\P}\big(\cshat_{\ell}(\PiPr{\rs}\qs)-\cshat_{\ell}(\qs)\big).
  \end{align*}
  which holds for all possible integers $\ell$.
  This implies that $\PiPr{\rs}\qs-\qs=0$, which is the assertion of
  the lemma.
  \ENDPROOF
\end{proof}

\begin{lemma}
  The polynomial projection $\PiPr{\rs}\vsh$ is
  \emph{computable} using only the degrees of freedom
  \TERM{D1}-\TERM{D3} of $\vsh\in\VhPrp{\rs}$.
\end{lemma}
\begin{proof}
  To prove the assertion of the lemma, we only need to prove that
  $\asP(\vsh,\qs)$
  and $(\vs,\us)_{\partial\P}=\int_{\partial\P}\vsh\qs\dS$
  are computable for all $\vsh\in\VhPrp{\rs}$ and scalar polynomial
  $\qs\in\PS{\rs}(\P)$.
  To this end, we integrate by parts $\asP(\vsh,\qs)$.
  For an odd $\pone$, i.e., $\pone=2\ell+1$, we find that
  \begin{align} 
    \asP(\vsh,\qs) 
    =
    & -\int_{\P} \big(\Delta^{\pone}\vsh)\,\qs\,\dx
    + \int_{\partial\P}\big(\partial_n\Delta^\ell\vsh\big)\,\Delta^\ell\qs\dS\nonumber \\[0.5em]
    &  +\sum_{i=1}^\ell
    \left( 
    \int_{\partial\P}\big(\partial_n\Delta^{\pone-i}\vsh\big)\,\Delta^{i-1}\qs\dS
    -\int_{\partial\P}\big(\Delta^{\pone-i}\vsh\big)\,\partial_n\Delta^{i-1}\qs\dS
    \right).
    \label{eq:poly:intbyparts:odd:p1}
  \end{align}
  For an even $\pone$, i.e., $\pone=2\ell$, we find that
  \begin{align} 
    \asP(\vsh,\qs) 
    &= \int_{\P}\big(\Delta^{\pone}\vsh\big)\,\qs\,\dx
    \nonumber\\[0.5em]
    &\phantom{=}
    -\sum_{i=1}^\ell
    \left(
    \int_{\partial\P}  \big(\partial_n\Delta^{\pone-i}\vsh\big)\,\Delta^{i-1}\qs\,\dS
    -\int_{\partial\P} \big(\Delta^{\pone-i}\vsh\big)\,\partial_n\Delta^{i-1}\qs\,\dS
    \right).
    \label{eq:poly:intbyparts:even:p1}
  \end{align}
  The first integral of the right-hand side of both formulas
  \eqref{eq:poly:intbyparts:odd:p1}
  and~\eqref{eq:poly:intbyparts:even:p1} is computable from the
  degrees of freedom \TERM{D3}.
  In turn, all the edge integrals are computable since we can expand
  the trace of $\Delta^{\mu}\vsh$ in terms of
  $\partial^{2(\mu-\nu)}_{t}\partial^{2\mu}_n\vsh$ and use the same
  argument of the proof of Lemma~\ref{lemma:unisolvence}.
  Since the edge traces of $\vsh$ and its normal derivatives (and all
  their tangential derivatives) are computable from the degrees of
  freedom of \TERM{D1}-\TERM{D2} through a polynomial interpolation,
  we deduce that all the edge integrals for both odd and even $\pone$
  and the boundary integral $(\vs,\us)_{\partial\P}$
  are computable.
  \ENDPROOF
\end{proof}

\subsection{Enhancement}

As noted in Remark~\ref{rem:ortho:proj:r-2p1}, the $L^2$-projection
operator $\PizP{\rs-2\pone}:\VhPrp{\rs}\to\PS{\rs}(\P)$ is computable
from the degrees of freedom \TERM{D3}.
Instead, to compute the orthogonal projection onto the polynomial
subspace $\PS{\rs-\pone}(\P)$, we need to modify the space definition
as follows thus obtaining the so called ``\emph{enhanced}'' virtual
element space.
Our construction follows the guidelines
in~\cite{Ahmad-Alsaedi-Brezzi-Marini-Russo:2013}.
First, we consider the mesh element $\P$ and the ``\emph{extended}''
virtual element space for $\rs\geq2\ptwo-1$ (recall that
$\ptwo\geq\pone$) defined as
\begin{multline}\label{eq:vem-space-higher:extended}
  \VhtPrp{\rs} := \Big\{
  \vsh\in\HS{\ptwo}(\P):\,
  \Delta^{\ptwo}\vsh\in\PS{\rs-\pone}(\P),\, \partial^j_n\vsh\in\PS{\rs-\js}(\E),\,\\
  \js=0,\ldots,\ptwo-1~\forall\E\in\partial\P
  \Big\}.
\end{multline}
Then, we define the enhanced virtual element space as
\begin{multline} 
  \WhPrp{\rs} := \bigg\{
  \vsh\in\VhtPrp{\rs}:\,
  \int_\P\vsh\qs\dx
  =
  \int_\P\PiPr{\rs-\pone}\vsh\qs\dx
  \\
  \quad \forall\qs\in\PS{\rs-\pone}\setminus\PS{\rs-2\pone}(\P)\bigg\}.
  \label{eq:enhanced:VEM:space}
\end{multline}
The polynomial space $\PS{\rs}(\P)$ is a subspace of $\WhPrp{\rs}$
and, thus, of $\VhtPrp{\rs}$, and the elliptic projection
$\PiPr{\rs-\pone}:\VhtPrp{\rs}\to\PS{\rs-\pone}(\P)$ that is defined
in~\eqref{eq:poly:Pi:A}-\eqref{eq:poly:Pi:B} is still computable and
only depends on the degrees of freedom \TERM{D1}, \TERM{D2} and
\TERM{D3}.
This assertion can easily be proved by repeating the argument of
Lemma~\ref{lemma:elliptic:projection}.

\medskip
The virtual element functions of the space $\VhtPrp{\rs}$ are uniquely
characterized by the set of degrees of freedom \TERM{D1}, \TERM{D2},
\TERM{D3} and the set of additional degrees of freedom
\TERM{\widetilde{D3}}
\begin{description}
\item\TERM{\widetilde{D3}}
  $\displaystyle\hP^{-2}\int_{\P}\qs\vsh\dx$ for any
  $\qs\in\PS{r-\pone}(\P)\setminus\PS{\rs-2\pone}(\P)$.
\end{description}
We state the unisolvence of these degrees of freedom in the following
lemma.
The proof is equal to the proof of Lemma~\ref{lemma:unisolvence}
(consider the degrees of freedom $(\TERM{D3},\TERM{\widetilde{D3}})$
instead of \TERM{D3}) and is omitted.
\begin{lemma}
  \label{lemma:unisolvence:enhanced}
  The degrees of freedom \TERM{D1}, \TERM{D2}, \TERM{D3},
  \TERM{\widetilde{D3}} are unisolvent in the virtual element space
  $\VhtPrp{r}$.
\end{lemma}

\begin{remark}
  According to Lemma~\ref{lemma:unisolvence:enhanced}, the dimension
  of $\VhtPrp{\rs}$ must be equal to the cardinality of the set of the
  degrees of freedom \TERM{D1}, \TERM{D2}, \TERM{D3},
  \TERM{\widetilde{D3}}.
  This statement can also be proved by a counting argument.
\end{remark}

Next, we want to prove that the degrees of freedom \TERM{D1},
\TERM{D2} and \TERM{D3} are unisolvent in the enhanced space
$\WhPrp{\rs}$.
To this end, we first need to establish a technical result.
Consider the set of linear, bounded functionals
$\lambda^{\tTERM{D1}}_{\ell_1},\lambda^{\tTERM{D2}}_{\ell_2},\lambda^{\tTERM{D3}}_{\ell_3}:\VhtPrp{\rs}\to\REAL$,
which respectively return the degrees of freedom \TERM{D1}, \TERM{D2}
and \TERM{D3} when applied to a virtual element function
$\vsh\in\VhtPrp{\rs}$.
The indices $\ell_1$, $\ell_2$, and $\ell_3$ run from $1$ to
$\#\TERM{D1}$, $\#\TERM{D2}$ and $\#\TERM{D3}$, respectively, where
$\#(\calD)$ denotes the cardinality of the discrete set $\calD$.
Renumbering $\ell_2$ and $\ell_3$ may require the introduction of
suitable sets of basis functions for the polynomial spaces
$\PS{\rs-\js}(\E)$ and $\PS{\rs-2\pone}(\P)$ in \TERM{D2} and
\TERM{D3}, respectively.
We left this aspect undefined as this technicality is not crucial in
this presentation, although important in the practical implementations
of the method.
We introduce the additional set of linear functionals
$\lambda^{\tTERM{\widetilde{D3}}}_{\ell}$ that are such that
\begin{align*}
  \lambda^{\tTERM{\widetilde{D3}}}_{\tell_3}(\vsh) =
  \hP^{-2}\int_{\P}\qs_{\tell_3}\big( \PiPr{\rs-\pone} - \PizP{\rs-\pone} \big)\vsh\dx
  \qquad\tell_{3}=1,\ldots,\#\widetilde{D3}
\end{align*}
where $\big\{\qs_{\tell_3}\big\}_{\tell_3}$ is a basis of
$\PS{\rs-\pone}(\P)\setminus\PS{\rs-2\pone}(\P)$, and the index
$\tell_3$ runs from $1$ to $\#\TERM{\widetilde{D3}}$, the number of
degrees of freedom of \TERM{\widetilde{D3}}.

\medskip
Then, we collect these different types of functionals in the
functional set
\begin{align}
  \Lambda = \big(\lambda_{\ell}\big) = \Big(
  \lambda^{\tTERM{D1}}_{\ell_1},
  \lambda^{\tTERM{D2}}_{\ell_2},
  \lambda^{\tTERM{D3}}_{\ell_3},
  \lambda^{\tTERM{\widetilde{D3}}}_{\tell_3}
  \Big).
\end{align}
We assume that the integer index $\ell$ is consistent with a suitable
renumbering of such degrees of freedom; so that $\ell$ runs form $1$
to $m'=m+\#\TERM{\widetilde{D3}}=\textrm{dim}\,\VhtPrp{\rs}$ and
$m=\#\TERM{D1}+\#\TERM{D2}+\#\TERM{D3}$.
These functionals satisfy the property stated in the following lemma.

\begin{lemma}\label{lm:independent}
  The linear functionals $\Lambda$ are linearly independent.
\end{lemma}
\begin{proof}
  Let $\vsh\in\VhtPrp{\rs}$ such that $\lambda_{\ell}(\vsh)=0$ for all
  $\ell=1,\ldots,m'$.
  Now, the degrees of freedom \TERM{D1}, \TERM{D2} and \TERM{D3} of
  $\vsh$ are (obviously) zero as they are the values of the
  functionals $\lambda^{\tTERM{D1}}_{\ell_1}(\vsh)$,
  $\lambda^{\tTERM{D2}}_{\ell_2}(\vsh)$ and
  $\lambda^{\tTERM{D3}}_{\ell_3}(\vsh)$.
  Moreover, it holds that $\PiPr{\rs}\vsh=0$, and, hence,
  $\PiPr{\rs-\pone}\vsh=0$, as these projections only depend on the
  degrees of freedom \TERM{D1}, \TERM{D2} and \TERM{D3},
  cf.~Lemma~\ref{lemma:elliptic:projection}.
  Then, we observe that the definition of the orthogonal projection
  $\PizP{\rs-\pone}\vsh$ and the facts that
  $\lambda^{\tTERM{\widetilde{D3}}}_{\tell_3}(\vsh)=0$ and
  $\PiPr{\rs-\pone}\vsh=0$ imply that
  \begin{align}
    \int_{\P}\qs\vsh\dx =
    \int_{\P}\qs\PizP{\rs-\pone}\vsh\dx =
    \int_{\P}\qs\PiPr{\rs-\pone}\vsh\dx = 0
  \end{align}
  for all $\qs\in\PS{\rs-\pone}(\P)\setminus\PS{\rs-2\pone}(\P)$.
  Therefore, the degrees of freedom \TERM{\widetilde{D3}} of $\vsh$
  are equal to zero and finally $\vsh=0$ because the degrees of
  freedom \TERM{D1}, \TERM{D2}, \TERM{D3} and \TERM{\widetilde{D3}}
  are unisolvent in $\VhtPrp{\rs}$.
  This argument proves that the intersection of the kernels of all the
  linear functionals $\lambda_{\ell}$ contains only the virtual
  element function that is identically zero over $\P$, so that these
  linear functionals are necessarily linearly independent.
  \ENDPROOF
\end{proof}

Using the linear functionals $\Lambda$, we reformulate the definition
of space $\WhPrp{\rs}$ in the following equivalent way:
\begin{multline}
  \hfill
  \WhPrp{\rs} := \Big\{
  \vsh\in\VhtPrp{\rs}:\,\lambda^{\tTERM{\widetilde{D3}}}_{\tell_3}(\vsh)=0
  \quad\forall\tell_3=m+1,\ldots,m'
  \Big\}.
  \hfill
  \label{eq:enhanced:VEM:space:new}
\end{multline}
In other words, $\WhPrp{\rs}$ belongs to the intersection of the
kernels of all the additional linear functionals $\lambda_{\ell}$ with
$\ell=m+1,\ldots,m'$.
The space $\WhPrp{\rs}$ has the two important properties that are
stated in the following lemma.
\begin{lemma}
  \label{lemma:WhP:property}
  The virtual element space $\WhPrp{\rs}$ has the same dimension of
  the ``regular'' space $\VhPrp{\rs}$  
  and the set of degrees of
  freedom \TERM{D1}, \TERM{D2} and \TERM{D3} are unisolvent in
  $\WhPrp{\rs}$.
\end{lemma}
\begin{proof}
  In view of Lemma~\ref{lm:independent} the linear functionals in
  $\Lambda$ are linearly independent and the cardinality of $\Lambda$,
  i.e., $m'=\#(\Lambda)$, is equal to the dimension of $\VhtPrp{\rs}$.
  Therefore, $\big(\P,\PS{\rs}(\P),\Lambda\big)$
  is a \emph{finite element} in the sense of Ciarlet,
  cf.~\cite{Ciarlet:2002}.
  So, there exists a set of $m'$ dual basis functions $\psi_{\ell}$
  such that
  \begin{align*}
    \lambda_{\ell}(\psi_{\ell'}) = \delta_{\ell,\ell'}
    \qquad\ell,\ell'=1,\ldots,m'.
  \end{align*}
  Now, it holds that
  $\lambda^{\tTERM{\widetilde{D3}}}_{\tell_3}(\psi_{\ell'})=\lambda_{\ell}(\psi_{\ell'})=0$
  for $\ell'=1,\ldots,m$, $\ell=m+1,\ldots,m'$ (and the corresponding
  values of the index $\tell_3$).
  This fact has the following consequences.
  The first $m$ linearly independent functions $\psi_{\ell'}$,
  $\ell'=1,\ldots,m$, belong to $\WhPrp{\rs}$, cf.
  formulation~\eqref{eq:enhanced:VEM:space:new}, since
  $\lambda^{\tTERM{\widetilde{D3}}}_{\tell_3}(\psi_{\ell'})=\lambda_{\ell}(\psi_{\ell'})=0$
  for $\ell=m+1,\ldots,m'$ (and corresponding indices $\tell_3$).
  This implies that $\textrm{dim}\,\WhPrp{\rs}\geq\ms$.
  Furthermore, according to the space definition
  \eqref{eq:enhanced:VEM:space:new} \emph{all} virtual element
  functions $\wsh\in\WhPrp{\rs}$ are such that
  $\lambda^{\tTERM{\widetilde{D3}}}_{\tell_3}(\wsh)=\lambda_{\ell}(\wsh)=0$
  for $\ell=m+1,\ldots,m'$.
  Therefore, such functions can be written as a linear combination of
  \emph{only} the first $m$ basis functions $\psi_{\ell'}$,
  $\ell'=1,\ldots,m$, are thus identified by the values of the linear
  functionals $\lambda^{\tTERM{D1}}_{\ell_1}\vsh$,
  $\lambda^{\tTERM{D2}}_{\ell_2}\vsh$ and
  $\lambda^{\tTERM{D3}}_{\ell_3}\vsh$.
  Consequently, all virtual element functions of $\WhPrp{\rs}$ are
  uniquely identified by the degrees of freedom \TERM{D1}, \TERM{D2},
  \TERM{D3} and, consequently, $\textrm{dim}\,\WhPrp{\rs}=\ms$.
  \ENDPROOF
\end{proof}

In view of Lemma~\ref{lemma:WhP:property}, the orthogonal projection
operator $\PizP{\rs-\pone}:\WhPrp{\rs}\to\PS{\rs-\pone}(\P)$ is
computable from the degrees of freedom \TERM{D1}, \TERM{D2} and
\TERM{D3}.

\medskip
Finally, we collect the local virtual element spaces into a
\emph{global} conforming virtual element space $\Whrp{r}$ defined
on $\Omega$ as
\begin{align}
  \Whrp{\rs} = \Big\{
  \wsh\in\HS{\pone}_{0}(\Omega)\cap\HS{\ptwo}(\Omega)\,:\,\restrict{\wsh}{\P}\in\WhPrp{\rs}\,\,\forall\P\in\Th
  \Big\},
  \label{eq:poly:global:space:enh}
\end{align}
where $\WhPrp{\rs}$ is the local space defined above.

\subsection{The virtual element bilinear form $\ash(\cdot,\cdot)$}

We discuss the definition of the bilinear form $\ash(\cdot,\cdot)$
that approximates the bilinear form $\as(\cdot,\cdot)$ in the virtual
element discretization~\eqref{eq:poly:VEM}.
This construction is the same for the ``regular'' virtual element
spaces~\eqref{eq:vem-space-lower} and~\eqref{eq:vem-space-higher} and
the ``enhanced'' space~\eqref{eq:enhanced:VEM:space}.
In this section we use the symbol $\VhPrp{\rs}$ to denote both choices
of the spaces.
However, such construction holds also for the enhanced virtual element
space~\eqref{eq:poly:global:space:enh}.
The symmetric bilinear form $\ash:\Vhrp{\rs}\times\Vhrp{\rs}\to\REAL$,
is written as the sum of local terms
\begin{align}
  \ash(\ush,\vsh) = \sum_{\P\in\Th}\ashP(\ush,\vsh),
\end{align}
where each local term $\ashP:\VhPrp{r}\times\VhPrp{r}\to\REAL$ is a
symmetric bilinear form.
We set
\begin{align}
  \ashP(\ush,\vsh) 
  = \asP(\PiPr{r}\ush,\PiPr{r}\vsh) 
  + \SP(\ush-\PiPr{r}\ush,\vsh-\PiPr{r}\vsh),
  \label{eq:poly:ah:def}
\end{align}
where $\SP:\VhPrp{r}\times\VhPrp{r}\to\mathbbm{R}$ provides the
stabilization term.
The stabilization form $\SP(\cdot,\cdot)$ is a symmetric, positive
definite bilinear form for which there exist two positive constants
$\sigma_*$ and $\sigma^*$ such that
\begin{align}
  \sigma_*\asP(\vsh,\vsh)\leq\SP(\vsh,\vsh)\leq\sigma^*\asP(\vsh,\vsh)
  \quad\forall\vsh\in\VhPrp{r}\textrm{~with~}\PiPr{r}\vsh=0.
  \label{eq:poly:S:stability:property}
\end{align}
The constants $\sigma_*$, $\sigma^*$ are independent of $\hh$ (and
$\P$).
A possible proof of the validity of
\eqref{eq:poly:S:stability:property} for the so called ``dofi-dofi''
stabilization in the context of arbitrarily regular conforming VEM can
be found in~\cite{Huang:2021} (for the case $\pone=2$ see
also~\cite{Beirao-Lovadina-Russo:2017,Chen-Huang:2018}).
This construction has the $\rs$-\emph{consistency} and
\emph{stability} properties stated in~\eqref{eq:poly:r-consistency}
and \eqref{eq:poly:stability},

\subsection{The virtual element approximation of the load term}\label{sec:rhs}

To approximate the right-hand side term of~\eqref{eq:poly:VEM} we
first assume the elemental decomposition
\begin{align}
  \bil{\fsh}{\vsh} = \sum_{\P\in\Th}\int_{\P}\fsh\vsh\,\dx.
  \label{eq:vem:rhs}
\end{align}
In Eq.~\eqref{eq:vem:rhs}, the elemental term $\restrict{\fsh}{\P}$ is
defined as
%
\begin{align}
  \restrict{\fsh}{\P} =
  \begin{cases}
    \Pizr{\rs-2\pone}\fs, & (a)~\textrm{if $\ptwo + 2\pone -1 \leq\rs$},\\[0.5em]
    \Pizr{\rs- \pone}\fs, & (b)~\textrm{if $\ptwo\leq\rs\leq\ptwo + 2\pone -2$}.
  \end{cases}
\end{align}
We discuss the two definitions of $\fsh$ given above separately.

\begin{remark}
  The right-hand side of \eqref{eq:vem:rhs} is fully computable by
  using only the degrees of freedom \TERM{D3} if $\rs\geq2\pone$ and
  we choose $\fsh$ as the piecewise polynomial approximation of $\fs$
  on $\Th$ in accordance with $(a)$.
  In such a case, we do not need to resort to the enhanced virtual
  element space defined in~\eqref{eq:enhanced:VEM:space}.
\end{remark}

Now, consider decomposition~\eqref{eq:vem:rhs} and definition $(a)$.
Since $\ptwo\geq\pone$, it holds that $\rs-2\pone\geq\ptwo-1$
(equivalently, $\rs\geq 3\pone-1$).
Thus, using the definition of the $\LTWO$-orthogonal projection, from
\eqref{eq:vem:rhs}, we find that
\begin{equation}
  \bil{\fsh}{\vsh} =
  \sum_{\P\in\Th}\int_{\P}\Pizr{\rs-2\pone}\fs\,\vsh\,\dx =
  \sum_{\P\in\Th}\int_{\P}\Pizr{\rs-2\pone}\,\fs\Pizr{\pone-1}\vsh\,\dx.
  \label{aux:1.1}
\end{equation}
Applying standard approximation results to~\eqref{aux:1.1} and
recalling that
$\vsh\in\Vhrp{\rs}\subset\HS{\pone}_{0}(\Omega)\cap\HS{\ptwo}(\Omega)$
yield the following estimate
\begin{align*}
  \bil{\fs - \fsh}{\vsh} \leq \Cs\hs^{\rs-\pone+1}
  \snorm{\vsh}{\pone}\snorm{\fs}{\rs-2\pone+1},
\end{align*}
for some positive constant $\Cs$ that is independent of $\hh$.
In particular, for $\pone=\ptwo\geq 2$ it is enough to choose
$r\geq2\ptwo+1$ (the case $\pone=\ptwo=2$ and $r\geq 5$ has been
originally treated in~\cite{Brezzi-Marini:2013}).
Note that for fixed values of $\pone$, larger values of the regularity
parameter $\ptwo$ ensure higher convergence rate for the approximation
of the right-hand side.
This is a specific attractive feature of arbitrarily regular
conforming VEM (which can not be exploited, e.g., in the nonconforming
setting).

\medskip
Now, consider decomposition~\eqref{eq:vem:rhs} and definition $(b)$.
Similarly to the previous case, using the definition of the
$\LTWO$-orthogonal projection yields
\begin{equation}
  \bil{\fsh}{\vsh} 
  = \sum_{\P\in\Th}\int_{\P}\Pizr{r-\pone}\fs\,\vsh\,\dx
  = \sum_{\P\in\Th}\int_{\P}\Pizr{r-\pone}\fs~\Pizr{0}\vsh\,\dx.
  \label{eq:aux:1:2}
\end{equation}
Applying again standard approximation results to~\eqref{eq:aux:1:2} we
we find that
\begin{align*}
  \bil{\fs - \fsh}{\vsh} \leq \Cs \hh^{\rs-\pone+2}
  \snorm{\vsh}{\pone} \snorm{\fs}{\rs-\pone+1}.
\end{align*}
For arbitrary values of $\pone$ and $\ptwo$, the use of the enhancement
approach might be avoided using arguments similar to those employed
in~\cite{Huang:2020}.

\subsection {Error analysis}
In this section, we briefly recall
a convergence
result in the energy norm \cite{Antonietti-Manzini-Verani:2019} (see also \cite{Antonietti-etal:2021,Huang:2021}) for the
approximation of \eqref{eq:poly:pblm:1}-\eqref{eq:poly:pblm:2}.
In particular, employing Theorem
\ref{theorem:poly:abstract:energy:norm} together with standard results of 
approximation (see, e.g., Reference
  \cite{BeiraodaVeiga-Brezzi-Marini-Russo:2016b,Beirao-Mora:2019,Huang:2021})
and the approximation properties of the right-hand
  side contained in Section~\ref{sec:rhs}.

\begin{theorem}
  \label{theorem:poly:energy:convg:rate}
   Let $\us\in H_0^{p_1}(\Omega)\cap H^{r+1}(\Omega)$
    be the solution of the polyharmonic
    problem~\eqref{eq:poly:pblm:1}-\eqref{eq:poly:pblm:2} and let $\ush\in\Vhrp{r}$ be
    the solution of the discrete problem~\eqref{eq:poly:VEM}.
    Assume that $\fs$ is  sufficiently regular. 
  Then, there exists a positive constant $C$ independent of $h$ such that
  \begin{align}
    \norm{\us-\ush}{\Vs}
    \leq\Cs\hh^{r-(p_1-1)}.
    \label{eq:poly:energy:convg:rate}
  \end{align}

\end{theorem}

\begin{remark}
Convergence estimates in lower order norms can be established provided
that classical duality arguments can be used and that the polynomial
approximation order $r$ is sufficient
large\cite{Chinosi-Marini:2016,Antonietti-Manzini-Verani:2019,Antonietti-etal:2021}.
\end{remark}

\section{Conclusion}
\label{sec6:conclusions}

We reviewed the construction of highly regular virtual element
spaces for the conforming approximations in two spatial dimensions
of elliptic problems of order $\pone\geq 1$.
The resulting finite dimensional virtual spaces are subspaces of
$\HS{\ptwo}(\Omega)$, $\ptwo\geq \pone$.
We presented an abstract convergence result in a suitably defined
energy norm. Moreover, after discussing the construction of the
approximation spaces and major aspects such as the choice and
unisolvence of the degrees of freedom, we provided specific examples
of highly regular virtual spaces, corresponding to various practical
cases.
Finally, a detailed discussion of the properties of the ``enhanced''
formulation of the virtual element spaces is provided.

\section*{Acknowledgments}
PFA, SVT and MV acknowledge the financial support of PRIN research
grant number 201744KLJL ``\emph{Virtual Element Methods: Analysis and
Applications}'' funded by MIUR.
GM acknowledges the financial support of the LDRD program of Los
Alamos National Laboratory under project number 20220129ER.
Los Alamos National Laboratory is operated by Triad National Security,
LLC, for the National Nuclear Security Administration of
U.S. Department of Energy (Contract No. 89233218CNA000001).
The Authors are affiliated to GNCS-INdAM (Italy).





\end{document}